\pdfoutput=1

\documentclass{amsart}%

\pdfoutput=1

\usepackage[T1]{fontenc}
\usepackage[mathscr]{eucal}% so-called Euler fonts, allows \mathscr
\usepackage{amssymb}% allows special arrows like >-> e.g.
\usepackage{booktabs} %pretty tables
\usepackage[usenames,dvipsnames]{xcolor} %changed to xcolor; color was giving ``incompatible color definition'' warnings
\usepackage[normalem]{ulem}% allows \sout to strike-out, cross-out text.
\usepackage{amsthm}% allows theorem* , e.g.
\usepackage{bbold}% allows \unit \mathbb{1}
\usepackage{comment}% allows multiple lines to be commeted out
\usepackage[perpage]{footmisc} %resets footnotes per page
\usepackage[inline]{enumitem}% allows resuming enumerate
\usepackage{amsmath}% allows pmatrix
\usepackage{tikz}
\usepackage{tikz-cd}
\usetikzlibrary{arrows}

\usepackage{etoolbox} %this is needed for the hack removing address indentation

\usepackage[unicode]{hyperref} %added unicode; was getting warnings otherwise

\definecolor{dark-red}{rgb}{0.5,0.15,0.15}
\definecolor{dark-blue}{rgb}{0.15,0.15,0.6}
\definecolor{dark-green}{rgb}{0.15,0.6,0.15}

\hypersetup{
    colorlinks, linkcolor=Blue,
    citecolor=Blue, urlcolor=Blue
}

\usepackage{comment}

\usepackage[nameinlink,capitalise,noabbrev]{cleveref}
\usepackage{microtype}

\numberwithin{equation}{section}% makes equat numb contain the section
\makeatletter
\renewcommand{\tocsection}[3]{%
  \indentlabel{\@ifnotempty{#2}{\makebox[1.5em][r]{\ignorespaces#1 #2.}\quad}}#3}
\makeatother
\usepackage[all]{xy}
\xyoption{line}
\newdir{ >}{{}*!/-10pt/\dir{>}}
\usepackage{graphicx}
\usepackage{mathtools}
\NewDocumentCommand\limder{e{_}}{\mathchoice
{\varprojlim  \IfValueT{#1}{_{\mathclap{#1}}}{}^{\!1\!}\mathop{}}
{\lim^1  \IfValueT{#1}{_{#1}}}
{\lim^1  \IfValueT{#1}{_{#1}}}
{\lim^1  \IfValueT{#1}{_{#1}}}}
\newtheorem{ThmAlpha}{Theorem}

\swapnumbers %pour que le numéro apparaisse devant le théorème

\newtheorem{Thm}[equation]{Theorem}
\newtheorem*{Thm*}{Theorem}
\newtheorem*{MainThm*}{Main Theorem}
\newtheorem{Prop}[equation]{Proposition}
\newtheorem{Lem}[equation]{Lemma}
\newtheorem{Cor}[equation]{Corollary}

\newtheorem*{Que*}{Question}
\newtheorem*{Goal*}{Goal}

\theoremstyle{remark}
\newtheorem{Def}[equation]{Definition}

\newtheorem{Not}[equation]{Notation}
\newtheorem{Ass}[equation]{Assumption}
\newtheorem{Exa}[equation]{Example}

\newtheorem{Rem}[equation]{Remark}

\newtheorem{Que}[equation]{Question}
\tikzset{
    labelrotatebelow/.style={anchor=north, rotate=90, inner sep=1.0mm}
}
\tikzset{
    labelrotateabove/.style={anchor=south, rotate=90, inner sep=1.0mm}
}

\newcommand{\nc}{\newcommand}
\nc{\dmo}{\DeclareMathOperator}

\nc{\Beren}[1]{{\color{MidnightBlue}#1}}
\nc{\Drew}[1]{{\color{Orange}#1}}
\nc{\Tobi}[1]{{\color{Green}#1}}
\nc{\Natalia}[1]{{\color{Yellow}#1}}
\nc{\Dout}[1]{\Drew{\sout{#1}}}
\nc{\Bout}[1]{\Beren{\sout{#1}}}
\nc{\Tout}[1]{\Tobi{\sout{#1}}}
\nc{\Nout}[1]{\Natalia{\sout{#1}}}
\nc{\Greg}[1]{{\color{magenta}#1}}

\usepackage{todonotes}

\usepackage{pdflscape}

\nc{\overbar}[1]{\mkern 1.5mu\overline{\mkern-1.5mu#1\mkern-1.5mu}\mkern 1.5mu}

\usepackage[a]{esvect}
\nc{\weaklyfinite}{weakly closed}
\nc{\finite}{closed}
\nc{\BCdual}[1]{{#1}^*}
\newcommand{\tring}{\mathrm{2}\text{-}\mathrm{Ring}}
\nc{\LCore}{\mathrm{LCore}}
\nc{\Stovicek}{\v{S}\v{t}ov\'{i}\v{c}ek}
\nc{\ftriple}{f_{\natural}}
\nc{\unitC}{\unit_{\cat C}}%
\nc{\unitD}{\unit_{\cat D}}%
\nc{\Pone}{{\mathbb{P}^1}}
\nc{\InvSupp}[1]{\Supp_{\cat T}^{-1}(#1)}%this is pretty bad notation; beren
\nc{\InvCosupp}[1]{\Cosupp_{\cat T}^{-1}(#1)}
\nc{\closureP}{\overbar{\{\cat P\}}}
\nc{\closureQ}{\overbar{\{\cat Q\}}}
\nc{\singP}{\{\cat P\}}
\nc{\singQ}{\{\cat Q\}}
\nc{\singm}{\{\frak m\}}
\dmo{\Inj}{Inj}
\dmo{\tfib}{tfib}
\dmo{\tcof}{tcof}
\dmo{\dual}{d}
\dmo{\Aut}{Aut}
\dmo{\tofib}{tofib}
\dmo{\surj}{surj}
\dmo{\Excs}{Exc}
\dmo{\Homog}{Homog}
\dmo{\PSh}{PSh}
\dmo{\Epis}{Epi}

\dmo{\KInjdmo}{K}
\dmo{\Dbdmo}{mod}
\dmo{\NilCoh}{NilCoh}
\dmo{\sur}{sur}
\nc{\KInj}[1]{\KInjdmo(\Inj #1)}
\nc{\Dbmod}[1]{\Der^b(\Dbdmo #1)}
\dmo{\Viss}{vis}%lol
\nc{\Vis}{\Viss}
\nc{\vis}{\Vis}
\nc{\kappaaux}{g}
\nc{\kappaCh}{{\kappaaux(\cat C_h)}}
\nc{\kappam}{{\kappaaux({\frak m})}}
\nc{\kappaP}{{\kappaaux(\cat P)}}
\nc{\kappaQ}{{\kappaaux(\cat Q)}}
\nc{\kappaCP}{{\kappaaux_{\cat C}(\cat P)}}
\nc{\kappaDP}{{\kappaaux_{\cat D}(\cat P)}}
\nc{\kappaCQ}{{\kappaaux_{\cat C}(\cat Q)}}
\nc{\kappaDQ}{{\kappaaux_{\cat D}(\cat Q)}}
\nc{\kappaphiB}{{\kappaaux(\phi(\cat B))}}
\nc{\kappaphiQ}{{\kappaaux(\varphi(\cat Q))}}
\nc{\halfplus}{{\scriptscriptstyle\top}}
\dmo{\Sub}{Sub}
\nc{\SpEn}{\cat S_{E(n)}}
\nc{\SpEnf}{\cat S_n}
\nc{\Lcomp}{L^{\mathrm{com}}} %I made this and the next one commands because I'm unsure of the choice of notation
\nc{\Ucomp}{U^{\mathrm{com}}}
\nc{\bbullet}{{\scriptscriptstyle\hspace{-1pt}\bullet}}
\nc{\bullett}{{\scriptscriptstyle\bullet}\hspace{-1pt}}
\nc{\LF}{L\hspace{-0.2ex}F}
\dmo{\StMod}{StMod}
\dmo{\Proj}{Proj}
\nc{\SpG}{\Sp^G}
\nc{\EG}{\bbE_G}
\nc{\DEG}{\Der(\EG)}
\nc{\DE}{\Der(\bbE)}
\nc{\Prst}{{\cat P}\mathrm{r^{st}}}
\nc{\Mack}[2]{\mathrm{Mack}_{#1}(#2)}
\nc{\SC}{S\cat C}
\dmo{\fin}{{fin}}
\dmo{\DM}{DM}
\dmo{\fp}{fp}
\nc{\DMQ}{\DM_Q}
\dmo{\DerKal}{DMack}
\dmo{\coh}{coh}
\dmo{\Der}{D}
\dmo{\DMot}{DMot}
\dmo{\rmH}{H}
\dmo{\piu}{\underline{\pi}}
\dmo{\Sphere}{\mathbb{S}}
\nc{\HA}{{\rmH \hspace{-0.2em}\bbA}}
\nc{\HZ}{{\rmH \hspace{-0.2em}\bbZ}}
\nc{\HZbar}{{\rmH \hspace{-0.2em}\underline{\bbZ}}}
\nc{\HbbF}{{\rmH \hspace{-0.15em}\mathbb{F}}}
\nc{\Fp}{{\bbF_{\hspace{-0.1em}p}}}
\nc{\HFp}{{\rmH \hspace{-0.15em}\bbF_{\hspace{-0.1em}p}}}
\nc{\HQ}{{\rm H \bbQ}}
\nc{\DHZG}{\Der(\HZ_G)}
\nc{\DHZH}{\Der(\HZ_H)}
\nc{\DHZK}{\Der(\HZ_K)}
\nc{\DHZGN}{\Der(\HZ_{G/N})}
\nc{\DHZGG}{\Der(\HZ_{G/G})}
\nc{\DHZCp}{\Der(\HZ_{C_p})}
\nc{\DHZGprime}{\Der(\HZ_{G'})}
\nc{\DHZ}{\Der(\HZ)}
\nc{\frakp}{\mathfrak{p}}
\nc{\frakq}{\mathfrak{q}}
\nc{\frakS}{\mathfrak{S}}
\nc{\frakT}{\mathfrak{T}}
\nc{\Z}{\mathbb{Z}}
\nc{\F}{\mathbb{F}}
\nc{\SSG}{\text{sSet}_*^G}
\nc{\sSet}{\text{sSet}}

\dmo{\csupp}{csupp}
\dmo{\Con}{Conj}
\dmo{\Id}{Id}
\dmo{\rmK}{\textrm{\rm K}}
\dmo{\Spc}{Spc}
\dmo{\thick}{thick}
\dmo{\thickid}{thickid}
\nc{\thicko}[1]{\thickid\langle #1 \rangle}
\nc{\thickt}[1]{\langle #1 \rangle}
\dmo{\Ann}{Ann}
\dmo{\cone}{cone}
\dmo{\End}{End}
\dmo{\Derperf}{D_{perf}}
\dmo{\Mor}{Mor}
\dmo{\id}{id}
\dmo{\incl}{incl}
\dmo{\Tor}{Tor}
\dmo{\Img}{Im}
\dmo{\im}{im}
\dmo{\Ker}{Ker}
\dmo{\ind}{ind}
\dmo{\CoInd}{coind}
\dmo{\GH}{GH}
\dmo{\idem}{e}
\dmo{\res}{res}
\dmo{\infl}{infl}
\dmo{\Derqc}{D_{qc}}
\nc{\DbcohX}{\Der^b(\coh X)}
\dmo{\triv}{triv}
\dmo{\Tel}{Tel} %telescope
\dmo{\grMod}{grMod}%
\dmo{\Mod}{Mod}%
\dmo{\opname}{op}
\dmo{\SH}{SH}% ground name for cat of spectra
\dmo{\smallb}{b}% ground exponent for ``bounded''
\dmo{\Spec}{Spec}
\dmo{\supp}{supp}
\dmo{\supph}{supph}
\dmo{\Supp}{Supp}
\dmo{\crosseffec}{cr}

\dmo{\thofib}{tofib}
\dmo{\hofib}{hofib}
\dmo{\cosupp}{cosupp}
\dmo{\Cosupp}{Cosupp}
\nc{\SHc}{{\SH^c}}
\nc{\SHp}{{\SH_{(p)}}}
\nc{\SHcp}{{\SH^c_{(p)}}}
\nc{\SHG}{\SH(G)}
\nc{\SHGp}{\SH(G)_{(p)}}
\nc{\SHGc}{\SHG^c}
\nc{\SHGcp}{\SHG^c_{(p)}}
\nc{\quadtext}[1]{\quad\textrm{#1}\quad}
\nc{\qquadtext}[1]{\qquad\textrm{#1}\qquad}
\nc{\adj}{\dashv}
\nc{\bbL}{\mathbb{L}}
\nc{\bbS}{\mathbb{S}}
\nc{\bbA}{\mathbb{A}}
\nc{\bbE}{\mathbb{E}}
\nc{\bbN}{\mathbb{N}}
\nc{\bbQ}{\mathbb{Q}}
\nc{\bbZ}{\mathbb{Z}}
\nc{\bbF}{\mathbb{F}}
\nc{\cat}[1]{\mathscr{#1}}%or: \nc{\cat}[1]{\mathcal{#1}}
\nc{\ie}{{\sl i.e.}, }
\nc{\into}{\mathop{\rightarrowtail}}
\nc{\inv}{^{-1}}
\nc{\isoto}{\mathop{\overset{\sim}\to}}
\nc{\isotoo}{\mathop{\overset{\sim}\too}}
\nc{\onto}{\mathop{\twoheadrightarrow}}
\nc{\too}{\mathop{\longrightarrow}\limits}
\nc{\mapstoo}{\longmapsto}
\nc{\adh}[1]{\overline{#1}}% adherence
\nc{\adhpt}[1]{\adh{\{#1\}}}% adherence of a pt
\nc{\aka}{{a.\,k.\,a.}\ }
\nc{\calF}{\mathcal{F}}
\nc{\eg}{{\sl e.\,g.}}
\nc{\hook}{\hookrightarrow}
\newcommand\noloc{%
  \nobreak
  \mspace{6mu plus 1mu}
  {:}
  \nonscript\mkern-\thinmuskip
  \mathpunct{}
  \mspace{2mu}
}
\nc{\ideal}[1]{\langle #1\rangle}
\dmo{\red}{red}
\usepackage{makecell}
\dmo{\Hom}{Hom}
\nc{\Homcat}[1]{\Hom_{\cat #1}}
\nc{\iHom}{\mathcal{H}\mathrm{om}}
\nc{\ihom}[1]{\mathsf{hom}(#1)}
\nc{\ihomC}[1]{\mathsf{hom}_{\cat C}(#1)}
\nc{\ihomD}[1]{\mathsf{hom}_{\cat D}(#1)}
\nc{\ihomsub}[2]{\mathsf{hom}_{#1}(#2)}
\usepackage{stmaryrd}
\usepackage{longtable}
\nc{\Mid}{\,\big|\,}
\nc{\MMod}{\,\text{-}\Mod}%
\nc{\GrMMod}{\,\text{-}\grMod}%
\nc{\op}{^{\opname}}
\nc{\oto}[1]{\overset{#1}\to}
\nc{\otoo}[1]{\overset{#1}{\,\too\,}}
\nc{\sminus}{\!\smallsetminus\!}
\nc{\poplus}[1]{^{\oplus #1}}%
\nc{\potimes}[1]{^{\otimes #1}}% tensor power
\nc{\sbull}{{\scriptscriptstyle\bullet}}%\mathbf{\cdot}}%{}}
\nc{\SET}[2]{\big\{\,#1\Mid#2\,\big\}}
\nc{\SpcK}{\Spc(\cat K)}% most used
\nc{\then}{\Rightarrow}
\nc{\unit}{\mathbb{1}}% unit for \otimes
\nc{\xra}{\xrightarrow}
\nc{\phigeom}[1]{\widetilde{\Phi}^{#1}}
\dmo{\Oname}{O}
\dmo{\proper}{proper}% for proper subgroups
\dmo{\lenormal}{\unlhd}
\dmo{\fib}{fib}
\dmo{\cofib}{cofib}
\dmo{\lnormal}{\lhd}
\nc{\normal}{\trianglelefteq}%\lhd
\nc{\Op}{\Oname^p}% O^p for maximal p-normal subgroup
\nc{\Oq}{\Oname^q}% as above for p=q
\newcommand{\Sp}{{\mathscr{S}p}}

\dmo{\Ho}{Ho}
\dmo{\Spf}{Spf}
\dmo{\CB}{CB}
\dmo{\Fin}{Fin}
\dmo{\add}{add}
\dmo{\Fun}{Fun}
\dmo{\Ext}{Ext}
\dmo{\CAlg}{CAlg}
\dmo{\CMon}{CMon}
\dmo{\CC}{\cat C} %beren: I changed these, but left the O
\dmo{\DD}{\cat D}
\dmo{\OO}{\mathcal{O}}
\dmo{\Map}{Map}
\dmo{\Span}{Span}
\dmo{\Tot}{Tot}
\dmo{\N}{N}
\dmo{\Cat}{Cat}
\dmo{\colim}{colim}
\dmo{\hocolim}{hocolim}
\dmo{\Ch}{Ch}
\dmo{\A}{\mathbb{A}^{eff}}
\nc{\AGeff}{\mathbb{A}_G^{\mathrm{eff}}}
\nc{\BGeff}{\mathcal{B}_G^{\mathrm{eff}}}
\nc{\BG}{{\mathcal{B}_G}}
\nc{\NBGeff}{{\N}{\BGeff}}
\dmo{\Ab}{Ab}
\nc{\Smith}{\mathsf{Smith}}
\nc{\Floyd}{\mathsf{Floyd}}
\nc{\blue}{\beth^{\mathrm{geom}}}
\dmo{\Set}{Set}
\dmo{\ev}{ev}
\dmo{\Spcl}{Spcl}
\nc{\Funadd}{\Fun_{\add}}
\dmo{\proj}{proj}
\dmo{\cof}{cof}
\nc{\cPd}{\cat P_{\hspace{-.1em}d}}
\nc{\cPm}{\cat P_{\hspace{-.1em}m}}
\nc{\Chp}{\mathsf{Ch}_p}
\nc{\Pp}{\mathsf{P}_p}

\dmo{\Coideal}{Coideal}
\dmo{\gen}{gen}
\nc{\auxcoidealsymb}{\vartriangleleft}
\dmo{\Loc}{Loc}
\dmo{\Ind}{Ind}
\dmo{\Coloc}{Coloc}
\dmo{\Locideal}{Locid}
\dmo{\Colocideal}{Colocid}
\nc{\LOCO}{\Locideal}
\nc{\COLOCO}{\Colocideal}
\nc{\Loco}[1]{\LOCO\langle #1 \rangle}
\nc{\Coloco}[1]{\COLOCO\langle #1 \rangle}
\nc{\LambdaP}{\Lambda^{\cat P}} %beren: I've added this command here because we might need to make some spacing changes to make the typesetting less ugly
\nc{\LambdaQ}{\Lambda^{\cat Q}} %beren: I've added this command here because we might need to make some spacing changes to make the typesetting less ugly
\nc{\GammaP}{\Gamma_{\cat P}} %beren: I've added this command here because we might need to make some spacing changes to make the typesetting less ugly
\nc{\GammaQ}{\Gamma_{\cat Q}} %beren: I've added this command here because we might need to make some spacing changes to make the typesetting less ugly
\nc{\LambdaW}{\Lambda^{\hspace{-0.3ex}W}} %beren: I've added this command here because we might need to make some spacing changes to make the typesetting less ugly
\nc{\GammaW}{\Gamma_{\hspace{-0.3ex}W}} %beren: I've added this command here because we might need to make some spacing changes to make the typesetting less ugly
\nc{\gW}{g_W}
\nc{\gP}{g_{\cat P}}
\nc{\gQ}{g_{\cat Q}}
\nc{\cC}{{\cat C}}
\nc{\cT}{{\cat T}}
\nc{\cD}{{\cat D}}

\nc{\mT}{\kern-0.5em\mod\kern-0.1em\text{-}\cat{T}^c}
\nc{\mTc}{\kern-0.5em\mod\kern-0.1em\text{-}\cat{T}^c}
\nc{\MTc}{\Mod\kern-0.1em\text{-}\cat{T}^c}
\nc{\MT}{\Mod\kern-0.1em\text{-}\cat{T}}
\newcounter{enum-resume-hack}
\usepackage{wasysym}%
\usepackage{caption}
\Crefname{Thm}{Theorem}{Theorems}
\Crefname{Prop}{Proposition}{Propositions}

\usepackage{makecell}
\usepackage{tablefootnote}
\usepackage{arydshln}
\usepackage{booktabs}
\usepackage{quiver}
\usepackage[all]{xy}
\newdir{ >}{{}*!/-10pt/\dir{>}}

\makeatletter
\providecommand*{\twoheadrightarrowfill@}{%
  \arrowfill@\relbar\relbar\twoheadrightarrow
}
\providecommand*{\twoheadleftarrowfill@}{%
  \arrowfill@\twoheadleftarrow\relbar\relbar
}
\providecommand*{\xtwoheadrightarrow}[2][]{%
  \ext@arrow 0579\twoheadrightarrowfill@{#1}{#2}%
}
\providecommand*{\xtwoheadleftarrow}[2][]{%
  \ext@arrow 5097\twoheadleftarrowfill@{#1}{#2}%
}
\makeatother

\nc{\cL}{\mathcal{L}}
\nc{\cP}{\mathcal{P}}

\nc{\tblue}{\beth^{\mathrm{Tate}}}

\nc{\cA}{\mathcal{A}}
\nc{\cF}{\mathcal{F}}
\nc{\Fnt}{\cF_{\mathrm{nt}}}
\nc{\Fall}{\cF_{\mathrm{all}}}
\nc{\Ftriv}{\cF_{\mathrm{triv}}}

\nc{\bE}{\underline{E}}

\usepackage{adjustbox}

\DeclareMathOperator{\Shv}{Shv}

\DeclareMathOperator{\Perf}{\operatorname{Perf}}%
\renewcommand{\epsilon}{\varepsilon}
\makeatletter
\renewcommand{\rightleftarrows}[2]{%
  \mathrel{\mathop{%
    \vcenter{\offinterlineskip\m@th
      \ialign{\hfil##\hfil\cr
        \hphantom{$\scriptstyle\mspace{8mu}{#1}\mspace{8mu}$}\cr
        \rightarrowfill\cr
        \noalign{\kern.3ex}% <— add space between arrows
        \leftarrowfill\cr
        \hphantom{$\scriptstyle\mspace{8mu}{#2}\mspace{8mu}$}\cr
        \noalign{\kern-0.3ex}
      }%
    }%
  }\limits^{#1}_{#2}}%
}
\makeatother
\let\tac\textasteriskcentered % two handy shortcut macros
\let\ted\textemdash
\newcommand\separ{%
   \begin{center}
     \ted\ted\ \tac\,\tac\,\tac\ \ted\ted%
   \end{center}}
\DeclareMathOperator{\qc}{qc}
\DeclareMathOperator{\qct}{qct}

\DeclareMathOperator{\Open}{Open}

\newcommand{\leftrarrows}{\mathrel{\raise.75ex\hbox{\oalign{%
  $\scriptstyle\leftarrow$\cr
  \vrule width0pt height.5ex$\hfil\scriptstyle\relbar$\cr}}}}
\newcommand{\lrightarrows}{\mathrel{\raise.75ex\hbox{\oalign{%
  $\scriptstyle\relbar$\hfil\cr
  $\scriptstyle\vrule width0pt height.5ex\smash\rightarrow$\cr}}}}
\newcommand{\Rrelbar}{\mathrel{\raise.75ex\hbox{\oalign{%
  $\scriptstyle\relbar$\cr
  \vrule width0pt height.5ex$\scriptstyle\relbar$}}}}

\DeclareMathOperator{\cofree}{cofree}

\Crefname{Thm}{Theorem}{Theorems}
\Crefname{Prop}{Proposition}{Propositions}
\Crefname{Lem}{Lemma}{Lemmas}
\Crefname{Cor}{Corollary}{Corollaries}
\Crefname{Exa}{Example}{Examples}
\Crefname{Rem}{Remark}{Remarks}

\AddToHook{env/Thm/begin}{     \crefalias{equation}{Thm}}
\AddToHook{env/Exa/begin}{     \crefalias{equation}{Exa}}
\AddToHook{env/Prop/begin}{    \crefalias{equation}{Prop}}
\AddToHook{env/Lem/begin}{     \crefalias{equation}{Lem}}
\AddToHook{env/Cor/begin}{     \crefalias{equation}{Cor}}
\AddToHook{env/Conj/begin}{    \crefalias{equation}{Conj}}
\AddToHook{env/Rem/begin}{    \crefalias{equation}{Rem}}
\AddToHook{env/Def/begin}{    \crefalias{equation}{Def}}
\title{The formal spectrum of a tensor-triangulated category}
\author{Drew Heard}
\author{Marius Nielsen}

\address{Drew Heard, Department of Mathematical Sciences, Norwegian University of Science and Technology, Trondheim}
\email{drew.k.heard@ntnu.no}
\urladdr{\href{https://folk.ntnu.no/drewkh/}{https://folk.ntnu.no/drewkh/}}

\address{Marius Nielsen, Department of Mathematical Sciences, Norwegian University of Science and Technology, Trondheim}
\email{marius.v.b.nielsen@ntnu.no}
\urladdr{\href{https://mariusnielsen.github.io/}{https://mariusnielsen.github.io/}}
\begin{document}
\begin{abstract}
To any essentially small tensor-triangulated category $\cat K$ and Thomason subset $Y \subseteq \Spc(\cat K)$ we associate a Dirac ringed space $(\Spf(\cat K,Y), \mathcal{O}_{\Spf(\cat K,Y)})$, called the \emph{formal spectrum} of $(\cat K,Y)$. We establish basic properties of this construction and compute it in several examples from algebraic geometry, chromatic homotopy theory, equivariant homotopy theory, and modular representation theory.
\end{abstract}
\maketitle

\tableofcontents

\section{Introduction}
Tensor-triangular geometry packages a tensor-triangulated category into a geometric object via the Balmer spectrum \cite{Balmer05a}. Concretely, if $\cat K$ is an essentially small tt-category (or a \emph{$2$-ring} in our terminology), then its Balmer spectrum $\Spc(\cat K)$ is a spectral space. Its topology has a basis of closed sets given by supports of objects in $\cat K$, and it carries a natural structure sheaf $\mathcal{O}_{\Spc(\cat K)}$ valued in commutative-graded rings, i.e., it is a Dirac ringed space in the sense of \cite{HesselholtPstragowskiDiracI}. When $\cat K=\Der_{\qc}(X)^c$ for a quasi-compact and quasi-separated (qcqs) scheme $X$, one recovers $X$ as a locally ringed space from $(\Spc(\cat K),\mathcal O_{\Spc(\cat K)})$; see \Cref{exa:balmer-scheme}. This is Balmer’s reconstruction theorem \cite{Balmer02} and \cite[Theorem 8.5]{BuanKrauseSolberg07}, relying crucially on Thomason’s classification of thick tensor ideals in $\Der_{\qc}(X)^c$ \cite{Thomason97}.

In classical algebraic geometry, passing from a scheme $X$ to its formal completion $\widehat{X}_Y$ along a closed subset $Y$ isolates the infinitesimal neighborhood of $Y$ and remembers all nilpotent thickenings at once. This is preferable whenever a problem is intrinsically local around $Y$, and it is natural to ask for an analogue in tensor–triangular geometry.  For example, the $K(n)$-local stable homotopy category is obtained from the $E(n)$-local category by a process resembling adic completion (see the formulas in \cite[Proposition~7.10]{HoveyStrickland99}), and it morally behaves like a ``residue field''. One might therefore expect its spectrum to be a single point. However, in tt-geometry one must work with essentially small subcategories, and the compact objects in the $K(n)$-local category do not contain the unit (hence do not form a tt-category) for $n>0$, while the spectrum of dualizable objects is, perhaps, unexpectedly large \cite{BarthelHeardNaumann20pp}.

The goal of this paper is therefore to develop a formal analogue of the Balmer spectrum. Given an essentially small tt-category\footnote{For technical reasons, we work with stable $\infty$-categories throughout; see \Cref{rem:why-infty-big}.} $\cat K$ and a Thomason subset $Y \subseteq \Spc(\cat K)$, we define the \emph{formal spectrum} $\Spf(\cat K,Y)$ to be a Dirac ringed space whose underlying topological space is $Y$, equipped with a completed analogue of Balmer’s structure sheaf; see \Cref{def:formal-balmer} for the complete definition. Roughly speaking, $\Spc(\cat K)$ plays the role of $\Spec(R)$, and $\Spf(\cat K,Y)$ is an analogue of the formal spectrum $\Spf(R,I)$ of a ring $R$ along an ideal $I$ (see \Cref{ref:def-formal-spectrum-affine} for our conventions on formal spectra of rings, which differ slightly from some standard sources).

In the classical case, it is straightforward to check that the completion map $R \to R^{\wedge}_I$ induces an isomorphism of formal schemes
\[
\Spf(R^{\wedge}_I,IR^{\wedge}_I) \xrightarrow{\sim} \Spf(R,I).
\]
An analogue holds in the tt-world, although the proof is more delicate. Since Bousfield localization is most naturally formulated in the presentable setting, it is convenient to pass from an essentially small $\cat K$ to the associated stable homotopy theory $\cat C \coloneqq \Ind(\cat K)$. Inside $\cat C$ we can form the completion at $Y\subseteq \Spc(\cat K)$, obtaining a full subcategory $\widehat{\cat C}_Y$ of $Y$-complete objects together with a symmetric monoidal localization $\widehat{(-)}_Y \colon \cat C \to \widehat{\cat C}_Y$. Unlike categories usually considered in tensor-triangular geometry, $\widehat{\cat C}_Y$ is typically not rigidly-compactly generated. Our first main result identifies the formal spectrum computed from the dualizable objects of $\widehat{\cat C}_Y$ with the original definition in terms of $(\cat K,Y)$.

\begin{ThmAlpha}[\Cref{thm:completion-theorem}]\label{thm:intro:a}
Let $\cat K$ be a $2$-ring and let $Y\subseteq \Spc(\cat K)$ be a Thomason subset.  There is a natural isomorphism of Dirac ringed spaces
\[
(\varphi,\varphi^\#)\colon
\Spf(\widehat{\cat C}^{\dual}_Y,\varphi^{-1}(Y))
\longrightarrow
\Spf(\cat K,Y),
\]
where $\varphi=\Spc(\widehat{(-)}_Y)$ is the map on Balmer spectra induced by completion.
\end{ThmAlpha}

This result is fundamental for us: the left-hand side is often the conceptual object of interest, while the right-hand side is typically easier to compute. For the remainder of this introduction, we abbreviate the left-hand side to $\Spf(\widehat{\cat C}^{\dual}_Y)$, suppressing the implicit Thomason subset $\varphi^{-1}(Y)$.

When $\Spc(\cat K)$ is noetherian, the formal spectrum admits a categorical refinement, in the following sense. 
\begin{ThmAlpha}[\Cref{thm:categorical-formal-sheaf,cor:categorical-formal-structure-sheaf}]\label{thm:categorical-formal-sheaf-intro}
Suppose that $\Spc(\cat K)$ is noetherian. There is a sheaf
\[
\cat O_{\Spf(\cat K,Y)}^{\infty}\colon\Open(Y)^{\op}\longrightarrow\CAlg(\Pr\nolimits_{\mathrm{st}}^L)
\]
such that, for every open subset $U \subseteq \Spc(\cat K)$, there is a natural symmetric monoidal equivalence
\[
\cat O_{\Spf(\cat K,Y)}^{\infty}(U \cap Y)\simeq\widehat{\cat C(U)}_{Y \cap U}.
\]
Moreover, applying endomorphisms of the tensor unit, taking homotopy groups sectionwise, and sheafifying recovers the structure sheaf $\mathcal O_{\Spf(\cat K,Y)}$.
\end{ThmAlpha}
\separ

A basic computation in tt-geometry is the Hopkins--Neeman theorem: for any commutative noetherian ring $R$ one has a canonical identification $\Spc(\Der(R)^c)\cong \Spec(R)$, under which tt-support agrees with ordinary (homological) support. Since formal completion in algebraic geometry is taken along a closed subset, we fix an ideal $I\subseteq R$ and consider the closed subset $V(I)\subseteq \Spec(R)\cong \Spc(\Der(R)^c)$. Our formal spectrum recovers the usual affine formal scheme:

\begin{ThmAlpha}[\Cref{thm:formal-hn}]\label{thm:formal-hn-intro}
Let $R$ be a commutative noetherian ring and $I \subseteq R$ an ideal. There is a natural isomorphism of formal schemes
\[
\rho \colon \Spf\bigl(\Der(R)^c,V(I)\bigr) \xrightarrow{\sim} \Spf(R,I).
\]
Consequently, the completion of $\Der(R)$ at $V(I)$ satisfies
\[
\Spf\bigl(\widehat{\Der(R)}^d_{V(I)}\bigr) \xrightarrow{\sim} \Spf(R^{\wedge}_I).
\]
\end{ThmAlpha}
The completed category $\widehat{\Der(R)}_{V(I)}$ here is the usual complete derived category of a noetherian ring at an ideal; see \Cref{exa:derived-complete-ring}.

We next extend this affine computation to noetherian schemes. The key input is a locality statement for the formal spectra with respect to restriction to quasi-compact opens. If $U\subseteq \Spc(\cat K)$ is quasi-compact with closed complement $Z$, recall that $\cat K(U)\coloneqq (\cat K/\cat K_Z)^{\natural}$ denotes the idempotent-completed Verdier quotient (cf.\ \Cref{rem:verdier-quotient}). Then we prove:

\begin{ThmAlpha}[\Cref{Thm:spf-open-restriction}]\label{Thm:spf-open-restriction-intro}
Let $\cat K$ be a $2$-ring and $Y\subseteq \Spc(\cat K)$ a Thomason subset. Let $U \subseteq \Spc(\cat K)$ be a quasi-compact open subset and let $q_U \colon \cat K\to \cat K(U)$ be the canonical functor. Set $Y'\coloneqq Y\cap U$, viewed as a Thomason subset of $\Spc(\cat K(U))\cong U$. Then restriction along $Y'\hookrightarrow Y$ induces an isomorphism of Dirac ringed spaces
\[
\Spf(\cat K(U),Y') \xrightarrow{\cong} \Spf(\cat K,Y)\big|_{Y'}.
\]
\end{ThmAlpha}
With \Cref{Thm:spf-open-restriction-intro} in hand, the affine computation globalizes by covering a noetherian scheme by finitely many affine opens and checking that the resulting local identifications glue on overlaps. This yields a formal version of Thomason's theorem.
\begin{ThmAlpha}[\Cref{thm:global-formal-hn}]\label{thm:global-formal-hn-intro}
Let $X$ be a noetherian scheme, and let $Z\subseteq X$ be a closed subset. Then there is a natural isomorphism of formal schemes
\[
\Phi \colon \Spf\bigl(\Der_{\qc}(X)^c,Z\bigr) \xrightarrow{\sim} \widehat{X}_Z,
\]
where $\widehat{X}_Z$ denotes the formal completion of $X$ along $Z$.
\end{ThmAlpha}

\separ

Our second main family of examples comes from chromatic homotopy theory. Fix a prime $p$ and a height $n\ge 0$, and write $\Sp_n$ for the $E(n)$-local stable homotopy category with compact objects $\Sp_n^c$. By work of Hovey--Strickland, the Balmer spectrum $\Spc(\Sp_n^c)$ is the finite chain of primes $\cat P_0 \supsetneq \cat P_1 \supsetneq \cdots \supsetneq \cat P_n$ indexed by chromatic height, and the structure sheaf on the basic opens can be described in terms of the localizations $L_kS^0$; see \Cref{thm:spec-en}. For each $0\le h\le n-1$ we consider the Thomason subset
\[
Y_{h+1}=\{\cat P_{h+1},\ldots,\cat P_n\}\subseteq \Spc(\Sp_n^c),
\]
corresponding to spectra of type at least $h{+}1$. Completion at $Y_{h+1}$ recovers the local category $\Sp_{K(h+1)\vee\cdots\vee K(n)}$, and we compute the following:

\begin{ThmAlpha}[\Cref{thm:formal-chromatic}]\label{thm:formal-chromatic-intro}
Fix an integer $0\le h\le n-1$. Then the formal spectrum $\Spf\bigl(\Sp_{K(h+1)\vee\cdots\vee K(n)}^{\dual}\bigr)$ has underlying space $Y_{h+1}$, and its structure sheaf on basic opens is given by
\[
\mathcal O(\overline U_k)\cong \pi_*L_{K(h+1)\vee\cdots\vee K(k)}S^0
\qquad (h+1\le k\le n),
\]
for the opens $\overline U_k=U(L_nF(k+1))\cap Y_{h+1}$ described in \Cref{thm:spec-en}.
\end{ThmAlpha}

These chromatic examples illustrate one of the motivations for introducing $\Spf(\cat K,Y)$: completion can drastically change the geometry detected by dualizable objects, and the formal spectrum provides a convenient receptacle for the resulting local information. In particular, in the special case where $h = n-1$, we see that the formal spectrum of dualizable objects in the $K(n)$-local category is indeed a single point, as one would hope. We conclude with two further examples, one from equivariant homotopy theory and one from modular representation theory, to illustrate additional behavior of the formal spectrum.
\medskip

\noindent\textbf{Organization of the paper.} In \Cref{sec:formal-schemes} we recall formal schemes and formal completions in algebraic geometry and fix conventions. In \Cref{sec:balmer} we review the Balmer spectrum and its structure sheaf. In \Cref{sec:completion} we recall completion of stable homotopy theories at Thomason subsets, and in \Cref{sec:formal-spectrum} we introduce the formal spectrum $\Spf(\cat K,Y)$ and prove the completion theorem \Cref{thm:completion-theorem}. \Cref{sec:categorical-sheaf} constructs a lift of the structure sheaf to the categorical level. In \Cref{sec:the_formal_hopkins_neeman_theorem,sec:thomason} we establish the formal version of the Hopkins--Neeman and Thomason theorems. \Cref{sec:a_comparison_map} produces a comparison map, analogous to Balmer's for the ordinary spectrum and in \Cref{sec:chromatic} we compute the formal spectra arising in chromatic homotopy theory. Finally, we finish in \Cref{sec:further_examples} with the two  examples from equivariant homotopy theory and modular representation theory.
\subsection*{Acknowledgments}
We thank Anish Chedalavada and Leovigildo Alonso Tarrio for helpful discussions. DH thanks Utrecht University and MPIM Bonn, where parts of this work were carried out.

\subsection*{Use of generative AI}
Generative AI was used for proofreading an earlier version of this paper, as well as help with some of the arguments in \Cref{sec:categorical-sheaf}. The authors take full responsibility for the arguments in this paper. 
\section{Formal schemes and formal completions}\label{sec:formal-schemes}
We begin by recalling the definition of the formal spectrum of a commutative noetherian ring. We do this to fix our conventions, which may differ slightly from the standard literature.

\begin{Def}\label{ref:def-formal-spectrum-affine}
Let $R$ be a commutative noetherian ring and let $I \subseteq R$ be an ideal. The \emph{formal spectrum} of $R$ (with respect to $I$), denoted $\Spf(R,I)$, is the locally ringed space
\[
(|\Spf(R,I)|, \mathcal{O}_{\Spf(R,I)})
\] 
where:
\begin{enumerate}
    \item The underlying topological space $|\Spf(R,I)|$ is the closed subset of the Zariski spectrum given by the vanishing locus of $I$, i.e.,
    \[
    |\Spf(R,I)| \coloneqq V(I) =  \SET{\mathfrak p \in  \Spec(R)}{I \subseteq \mathfrak p}.
    \]
    It has a basis of open subsets of the form
    \[
    \overline{D}(s) \coloneqq D(s) \cap V(I) = \SET{\mathfrak p \in \Spec(R)}{ I \subseteq \mathfrak p,\, s \notin \mathfrak p},
    \]
    where $s \in R$ and
    \[
    D(s) = \SET{\mathfrak p \in \Spec(R)}{s \notin \mathfrak p}.
    \]

    \item The structure sheaf $\mathcal{O}_{\Spf(R,I)}$ is the unique sheaf on $V(I)$ whose value on each basic open set $\overline{D}(s)$ is the $I$-adic completion of the localization $R[1/s]$.
\end{enumerate}
\end{Def}

\begin{Exa}\label{exa:nilpotent-ideal}
Let $R$ be a noetherian ring, and let $I$ be a nilpotent ideal (for example, $I = (0)$). Then it is immediate from the definitions that
\[
\Spf(R,I) \cong \Spec(R),
\] 
the usual Zariski spectrum.
\end{Exa}
\begin{Exa}
Let $R = \mathbb{Z}$ and $I = (p)$ for a prime $p$. Then $\Spf(\mathbb{Z},(p))$ is a one-point locally ringed space with ring of global sections $\mathbb Z_p$.
\end{Exa}
Since the rings arising as homotopy groups below are naturally Dirac rings, we now record the corresponding graded conventions.
\begin{Def}\label{def:dirac-ringed-space}
Following Hesselholt--Pstr\k{a}gowski \cite{HesselholtPstragowskiDiracI}, a \emph{Dirac ring} is a commutative algebra object in graded abelian groups equipped with the Koszul symmetric monoidal structure. Thus, for homogeneous elements $x$ and $y$, one has $xy=(-1)^{|x||y|}yx$. Ideals, prime ideals, and maximal ideals in a Dirac ring are understood to be homogeneous. A Dirac ring is \emph{local} if it has a unique maximal homogeneous ideal.

A \emph{Dirac ringed space} is a topological space $X$ equipped with a sheaf $\mathcal O_X$ of Dirac rings. It is \emph{locally Dirac ringed} if every stalk $\mathcal O_{X,x}$ is a local Dirac ring. A morphism of Dirac ringed spaces consists of a continuous map together with a morphism of sheaves of Dirac rings; it is a morphism of locally Dirac ringed spaces if the induced maps on stalks are local. Ordinary (locally) ringed spaces embed fully faithfully by placing their structure sheaves in degree zero.
\end{Def}

\begin{Def}\label{def:formal-dirac-spectrum}
Let $R_*$ be a noetherian Dirac ring, meaning that every graded ideal is generated by finitely many homogeneous elements, and let $I\subseteq R_*$ be a graded ideal.

For a graded $R_*$-module $M_*$, define its $I$-adic completion by
\[
(M_*)^{\wedge}_I\coloneqq\varprojlim_n M_*/I^nM_*.
\]
The inverse limit is taken in the category of graded $R_*$-modules and is calculated degreewise on the underlying graded abelian groups. When $M_*=R_*$, this completion carries its canonical Dirac ring structure.

Let $\Spec^h(R_*)$ denote the homogeneous Zariski spectrum of $R_*$. The \emph{homogeneous formal spectrum} of $R_*$ along $I$, denoted $\Spf^h(R_*,I)$, is the Dirac ringed space whose underlying space is $V(I)\subseteq\Spec^h(R_*)$ and whose structure sheaf is obtained by sheafifying the assignment
\[
\overline D(s)\longmapsto (R_*[1/s])^{\wedge}_I
\]
for homogeneous $s\in R_*$. 
\end{Def}

\begin{Rem}\label{rem:adic-completion}
We do not assume that $R$ (or $R_*$ in the graded case) is complete with respect to the $I$-adic topology. However, replacing $R$ by its completion $R^{\wedge}_I$ does not change the formal spectrum; there is an isomorphism of locally ringed spaces
\[
\Spf(R^{\wedge}_I, IR^{\wedge}_I) \xrightarrow{\sim} \Spf(R,I)
\]
induced by the completion map $R \to R^{\wedge}_I$. Likewise, for a noetherian Dirac ring $R_*$ and a homogeneous ideal $I\subseteq R_*$, completion induces an isomorphism of Dirac ringed spaces
\[
\Spf^h\bigl((R_*)^{\wedge}_I,I(R_*)^{\wedge}_I\bigr)\xrightarrow{\sim}\Spf^h(R_*,I).
\]
In either case, since the ideal for the formal spectrum is clear, we may simply write $\Spf(R^{\wedge}_I)$ or $\Spf^h((R_*)^{\wedge}_I)$.
\end{Rem}

\begin{Rem}
Although the definition of the formal spectrum of a ring $R$ does not strictly require $R$ to be noetherian, we fix this assumption due to the poorly behaved nature of $I$-adic completion in the non-noetherian case.
\end{Rem}

\begin{Rem}\label{rem:support-of-a-module}
There is a useful way to think about the topological space $|\Spf(R,I)|$. Recall that the support of a (finitely generated) $R$-module $M$ is
\[
\Supp_R(M) = \SET{\mathfrak p \in \Spec(R)}{M_{\mathfrak p} \ne 0}.
\]
Since
\[
S^{-1}(R/I) = 0 \iff S \cap I \ne \varnothing,
\] 
we see that
\[
\Supp_R(R/I) = \SET{\mathfrak p \in \Spec(R)}{I \subseteq \mathfrak p}.
\]
In other words,
\[
|\Spf(R,I)| = \Supp_R(R/I).
\]
\end{Rem}

In the same way that the spectrum of a ring globalizes to define schemes, the formal spectrum globalizes to define formal schemes.

\begin{Def}\label{def:formal-scheme}
Let $(X,\mathcal{O}_X)$ be a locally ringed space. We say that $(X,\mathcal{O}_X)$ is a \emph{noetherian formal scheme} if $|X|$ is quasi-compact and there exists a covering of $|X|$ by open subsets $U_{\alpha}$ such that each ringed space $(U_{\alpha}, \mathcal{O}_{X}|_{U_{\alpha}})$ is isomorphic to
\[
(|\Spf(R_{\alpha}, I_{\alpha})|, \mathcal{O}_{\Spf(R_{\alpha},I_{\alpha})})
\] 
for some noetherian ring $R_{\alpha}$ and ideal $I_{\alpha} \subseteq R_{\alpha}$.
\end{Def}
\begin{Rem}
It is clear from the definition that the condition of being a noetherian formal scheme is local with respect to finite covers. More specifically, if there exists a finite covering $\{ U_{\alpha} \}$ of $X$ such that each $(U_{\alpha}, \mathcal{O}_{X}|_{U_{\alpha}})$ is a noetherian formal scheme, then so is $(X,\mathcal{O}_X)$.
\end{Rem}

\begin{Rem}
We form the category of noetherian formal schemes as the full subcategory of the category of locally ringed spaces whose objects are noetherian formal schemes. In contrast to some other sources, we do not require the structure sheaf to carry a topological structure (i.e., we work with locally ringed spaces rather than topologically ringed spaces).
\end{Rem}

\section{The Balmer spectrum of a tt-category}\label{sec:balmer}
We assume the reader has some familiarity with ordinary tt-geometry, for example, the foundational paper \cite{Balmer05a}. For the benefit of the reader, we briefly recall the basic definitions. We begin by introducing the `small' and `big' variants of triangulated categories that will be used; see also \cite[Section~5]{BCHNPS-descent} for a more thorough discussion.

\begin{Def}\label{def:2ring}
Let
\[
  \tring \coloneqq \CAlg(\Cat^{\mathrm{perf}}_{\infty})_{\mathrm{rig}}
\]
be the $\infty$-category of commutative algebra objects in $\Cat^{\mathrm{perf}}_{\infty}$, the $\infty$-category of small, stable, idempotent-complete $\infty$-categories and exact functors, equipped with the symmetric monoidal structure described in \cite[Section~3.1]{BlumbergGepnerTabuada2013universal}. The subscript $\mathrm{rig}$ denotes the full subcategory spanned by those $\cat K$ for which every object of the underlying symmetric monoidal $\infty$-category is dualizable. A \emph{2-ring} is then an object $\cat K \in \tring$.
\end{Def}

We will also need a `big' variant of this notion.

\begin{Def}
A \emph{stable homotopy theory} is an object of $\CAlg(\Pr^L_{\mathrm{st}})$, i.e., a presentable, symmetric monoidal stable $\infty$-category $(\cat C, \otimes, \unit)$ whose tensor product preserves all colimits separately in each variable. The full subcategory $\cat K \coloneqq \cat C^{\dual}$ of dualizable objects is then a $\tring$ in the sense of \Cref{def:2ring} (see \cite[Lemma~2.5]{NaumannPol2024Separable}). A stable homotopy theory is \emph{rigidly-compactly generated} if it is compactly generated and its subcategory of dualizable objects coincides with its subcategory of compact objects.
\end{Def}
\begin{Rem}\label{rem:why-infty-big}
The homotopy category of a given $\cat K \in \tring$ is a tensor-triangulated category in the usual sense of \cite{Balmer05a}. Likewise, if $\cat C$ is a stable homotopy theory, then its homotopy category is a tensor-triangulated category with small coproducts.

For the definition of the Balmer spectrum $\Spc(\cat K)$ it suffices to work at the level of these underlying triangulated categories. By contrast, our construction of the formal spectrum as a \emph{Dirac ringed space} is most naturally formulated after passing to ``big'' categories, as we will soon describe. We therefore work throughout in the setting of stable $\infty$-categories: among other advantages, Ind-completion is well behaved in this setting, whereas on the level of triangulated categories the Ind-completion need not remain triangulated; see \cite[Remark~5.9]{krause2023completionstriangulatedcategories} for an example.
\end{Rem}
\begin{Rem}\label{Rem:ind-completion}
By taking Ind-categories, any 2-ring can be turned into a rigidly-compactly generated stable homotopy theory. Indeed, set $\cat C \coloneqq \Ind(\cat K)$. As noted in \cite{NaumannPol2024Separable}, this is presentable, stable, and symmetric monoidal by combining \cite[Theorem~5.5.1.1]{HTTLurie} and \cite[Corollary~4.8.1.14]{HALurie}, and it is rigidly-compactly generated by construction. Moreover, the (restricted) Yoneda embedding $j \colon \cat K \to \cat C$ is symmetric monoidal. In fact, if $\Pr^{L,\omega}_{\mathrm{st}}$ denotes the (non-full) symmetric monoidal subcategory of $\Pr^L_{\mathrm{st}}$ spanned by compactly generated stable $\infty$-categories and colimit-preserving functors that preserve compact objects (equivalently, whose right adjoint preserves colimits), then $\Ind$-completion defines a symmetric monoidal equivalence
\[
   \Ind \colon \Cat^{\mathrm{perf}}_{\infty} \xrightarrow{\sim} \Pr\nolimits^{L,\omega}_{\mathrm{st}},
\]
with inverse given by sending $\cat D$ to $\cat D^{\omega}$. See \cite[Section~3.1]{BlumbergGepnerTabuada2013universal} and \cite[Proposition 2.3]{BenMosheCarmeliSchlankYanovski2025Descent}. We will use this equivalence implicitly whenever needed.
\end{Rem}

Having fixed terminology for both the small and big settings, we now give a standard example that will serve as a running model in what follows.

\begin{Exa}
The prototypical example of a 2-ring is the category $\Mod^c_R$ of compact $R$-modules, where $R$ is a commutative ring spectrum. The associated stable homotopy theory is the category $\Mod_R$ of all $R$-modules. Note that if $R = HA$ is an Eilenberg--MacLane spectrum, these correspond to the derived category of perfect complexes of $A$-modules and the unbounded derived category of $A$-modules, respectively.
\end{Exa}

\begin{Def}[Balmer \cite{Balmer05a}]
Let $\cat K \in \tring$. The \emph{spectrum} of $\cat K$ is the set
\[
  \Spc(\cat K) \coloneqq \{\, \cat P \mid \cat P \text{ is a prime thick } \otimes\text{-ideal of } \cat K \,\}.
\]
We topologize this set by declaring a basis of open sets $\{\,U(c)\,\}_{c\in\cat K}$ defined by
\[
  U(c) \coloneqq \SET{\cat P \in \Spc(\cat K)}{c \in \cat P}.
\]
\end{Def}

\begin{Def}
The \emph{support} of an object $a \in \cat K$ is
\[
  \supp(a) \coloneqq \SET{\cat P \in \Spc(\cat K)}{a \notin \cat P} = U(a)^{\complement}.
\]
The sets $\{\supp(a)\}_{a\in\cat K}$ form a basis of closed subsets for the topology on $\Spc(\cat K)$.
\end{Def}

\begin{Rem}\label{rem:functoriality}
The spectrum is functorial \cite[Proposition~3.6]{Balmer05a}: given a symmetric monoidal exact functor $F \colon \cat K \to \cat L$, there is an induced continuous map
\[
\varphi \coloneqq \Spc(F) \colon \Spc(\cat L) \longrightarrow \Spc(\cat K), \qquad 
\cat P \longmapsto F^{-1}(\cat P).
\]
Moreover, $\varphi^{-1}(\supp_{\cat K}(a)) = \supp_{\cat L}(F(a))$ and $\varphi^{-1}(U(c)) = U(F(c))$.
\end{Rem}

One major reason to study the spectrum is its classification of thick tensor ideals. Here, we recall that a Thomason subset is a union of closed subsets, each with quasi-compact complement. 

\begin{Thm}[Balmer, {\cite[Theorem~4.10]{Balmer05a}}]\label{thm:balmer}
Let $\cat K$ be a 2-ring. Then support induces an inclusion-preserving bijection:
\[
\begin{aligned}
\Bigl\{\text{thick $\otimes$-ideals }\cat J\subseteq\cat K\Bigr\}
&\leftrightarrows
\Bigl\{\text{Thomason subsets }Y\subseteq\Spc(\cat K)\Bigr\} \\
\cat J &\longmapsto \supp(\cat J) \coloneqq \bigcup_{x\in\cat J}\supp(x) \\
\cat K_Y \coloneqq \{\,x\in\cat K \mid \supp(x)\subseteq Y\,\} &\longmapsfrom Y.
\end{aligned}
\]
\end{Thm}

\begin{Rem}\label{rem:verdier-quotient}
We now describe Balmer's structure sheaf on $\Spc(\cat K)$. Given a thick $\otimes$-ideal $\cat J \subseteq \cat K$ as above, we can form the Verdier quotient $\cat K / \cat J$. If $U \subset \Spc(\cat K)$ is a quasi-compact open subset with closed complement $Z$, set
\[
\cat K_Z \coloneqq \SET{a \in \cat K}{\supp(a) \subseteq Z}
\]
and define
\[
\cat K(U) \coloneqq \big( \cat K / \cat K_Z \big)^{\natural},
\]
where $(-)^\natural$ denotes idempotent completion.  This is also known as the Karoubi quotient. There is a natural functor $q_U \colon \cat K \to \cat K(U)$, and we write $\unit_U$ for the image of the unit $\unit \in \cat K$ under $q_U$.
\end{Rem}

\begin{Def}
Let $\cat K\in\tring$ be a $2$-ring, and let $\cat B$ denote the basis of
quasi-compact open subsets of $\Spc(\cat K)$. Define a presheaf
\footnote{Here, and throughout, $\mathrm{gr}\!\CAlg^{\heartsuit}$ denotes
the category of Dirac rings from \Cref{def:dirac-ringed-space}.}
\[
{}_{\mathrm p}\mathcal O_{\Spc(\cat K)}
\colon \cat B^{\op}\longrightarrow\mathrm{gr}\!\CAlg^{\heartsuit}
\]
by
\begin{equation}\label{eq:graded-balmer-sheaf}
{}_{\mathrm p}\mathcal O_{\Spc(\cat K)}(U)
\coloneqq \pi_*\End_{\cat K(U)}(\unit_U).
\end{equation}
Balmer's structure sheaf is the associated sheafification
\[
\mathcal O_{\Spc(\cat K)}
\coloneqq a\bigl({}_{\mathrm p}\mathcal O_{\Spc(\cat K)}\bigr)
\]
on $\Spc(\cat K)$.
\end{Def}

\begin{Rem}\label{rem:verdier-quotients-for-small-and-large-categories}
Let $\cat K$ be a 2-ring and $\cat C \coloneqq \Ind(\cat K)$ the associated stable homotopy theory. Let $U \subseteq \Spc(\cat K)$ be a quasi-compact open subset with closed complement $Z$. Consider the localizing subcategory $\cat L \coloneqq \Loc\langle \cat K_Z \rangle \subseteq \cat C$, and set $\cat C(U) \coloneqq \cat C/\cat L$. Since $\cat K_Z$ is a thick $\otimes$-ideal in $\cat K$ and the tensor product in $\cat C$ preserves colimits separately in each variable, $\cat L$ is a localizing $\otimes$-ideal in $\cat C$. Consequently, there is a unique way to equip $\cat C(U)$ and the localization functor $L\colon \cat C\to \cat C(U)$ with symmetric monoidal structures.

We now identify the compact objects of the localization. The restriction of $L$ to dualizable (equivalently, compact) objects, $L|_{\cat K} \colon \cat K \to \cat C(U)$, lands in $\cat C(U)^c$ (since $L$ preserves compact objects) and vanishes on $\cat K_Z$. By the universal property of the Verdier quotient, this induces a symmetric monoidal functor
\[
\bar{L} \colon \cat K / \cat K_Z \longrightarrow \cat C(U)^c.
\]
The Neeman--Thomason localization theorem \cite[Theorem~2.1]{Neeman92b} \& \cite[Theorem~A.3.12]{MR5009505}\footnote{The former is stated for triangulated categories and the latter for stable $\infty$-categories.} asserts that $\bar{L}$ is fully faithful and that $\cat C(U)^c$ identifies with the idempotent completion of the image of $\bar{L}$. Therefore, passing to idempotent completions yields a symmetric monoidal equivalence
\[
\cat K(U) \coloneqq (\cat K / \cat K_Z)^{\natural} \xrightarrow{\sim} (\cat C(U)^c)^{\natural} \simeq \cat C(U)^c.
\]
\end{Rem}
\begin{Rem}\label{remark:structure-sheaf-from-hom-functor}
If $\unit_U$ denotes the tensor unit of the localized category $\cat C(U)$, then the presheaf from \eqref{eq:graded-balmer-sheaf} can be rewritten as
\begin{equation}\label{eq:complicated-balmer-sheaf}
  U \longmapsto \pi_*\Hom_{\cat C}(\unit, \unit_U),
\end{equation}
where in the right-hand side, $\unit_U$ is implicitly viewed inside $\cat C$ via the fully faithful right adjoint to the localization functor $L$. Although \eqref{eq:complicated-balmer-sheaf} looks slightly less direct than \eqref{eq:graded-balmer-sheaf}, this ambient formulation in $\cat C$ will be crucial for the construction of the formal Balmer spectrum, where we will encounter localizations that do not preserve compact objects, necessitating the use of the large category $\cat C$.
\end{Rem}
\begin{Rem}\label{remark:sheaf-of-categories}
In \cite[Theorem~D and Theorem~F]{aoki2025higherzariskigeometry}, it is shown that for any $2$-Ring $\cat K$ with associated stable homotopy theory $\cat C \coloneqq \Ind(\cat K)$, the functor sending a quasi-compact open $U\subseteq \Spc(\cat K)$ to the $\infty$-category $\cat C(U)$ refines to a sheaf
  \[
    \mathcal{O}_{\cat K} \colon \Open(\Spc (\cat K))^{\mathrm{op}} \longrightarrow \CAlg(\Pr\nolimits^L_{\mathrm{st}}).
  \]
\end{Rem}
\begin{Rem}\label{remark:structure-sheaf-of-ring-spectra}
It follows from \cite[Theorem~4.8.5.11]{HALurie} and \cite[Theorem~4.8.5.16]{HALurie} that the functor
    \[
    \CAlg(\Sp) \longrightarrow \CAlg(\Pr\nolimits_{\mathrm{st}}^L),
    \]
    sending a commutative algebra $A$ in spectra to the stable $\infty$-category $\Mod_A(\Sp)$ of $A$-module spectra, is symmetric monoidal and admits a right adjoint:
    \begin{equation}\label{eq:end-functor}
    \End(\unit_{(-)}) \colon \CAlg(\Pr\nolimits_{\mathrm{st}}^L) \longrightarrow \CAlg(\Sp),
    \end{equation}
    sending a stable homotopy theory $\cat C$ to the commutative algebra in spectra given by the endomorphism spectrum $\Hom_{\cat C}(\unit_{\cat C},\unit_{\cat C})$ (see also \cite[Construction 4.29]{aoki2025higherzariskigeometry}).

    If we compose $\mathcal{O}_{\cat K}$ with $\End(\unit_{(-)})$ we obtain a sheaf
    \[
    \mathcal{O}_{\cat K}^{\unit} \colon \Open(\Spc(\cat K))^{\mathrm{op}} \longrightarrow \CAlg(\Sp)
    \]
    of commutative algebras in spectra. This sheaf is given on a quasi-compact open $U\subseteq \Spc(\cat K)$ by $\Hom_{\cat C(U)}(\unit_{U},\unit_{U})\simeq \Hom_{\cat C}(\unit,\unit_{U})$ in the notation of \cref{rem:verdier-quotients-for-small-and-large-categories}.

    Consequently, Balmer's structure sheaf of Dirac rings is obtained by applying $\pi_*$ sectionwise and then sheafifying:
    \[
    \mathcal O_{\Spc(\cat K)}\simeq a\bigl(\pi_*\mathcal O_{\cat K}^{\unit}\bigr),
    \]
where $a$ denotes sheafification in Dirac rings.
\end{Rem}
\begin{Rem}[Functoriality]
The functoriality of \Cref{rem:functoriality} is compatible with this Dirac ringed space structure. That is, given a tt-functor $F \colon \cat K \to \cat L$, there is a morphism of Dirac ringed spaces
\[
    \Spc(F) \coloneqq \varphi \colon (|\Spc(\cat L)|,\mathcal{O}_{\Spc(\cat L)}) \to (|\Spc(\cat K)|,\mathcal{O}_{\Spc(\cat K)}).
\]
For details, see \cite[Lemma 7.2]{BuanKrauseSolberg07}.
\end{Rem}
\begin{Exa}\label{exa:balmer-scheme}
The most fundamental example comes from work of Hopkins and Neeman in the affine case, as extended to schemes by Thomason \cite{Hopkins87,Neeman92a,Thomason97}. Let $X$ be a quasi-compact and quasi-separated scheme. There is an isomorphism of locally ringed spaces (and hence schemes)
\[
f \colon X \xrightarrow{\sim} \Spc(\Der_{\qc}(X)^c)
\]
given by
\[
f(x) = \SET{a \in \Der_{\qc}(X)^c}{a_x \simeq 0 \text{ in } \Der(\mathcal{O}_{X,x})^c} \quad \text{ for all } x \in X.
\]
Moreover, under this equivalence, the homological support\footnote{That is, the support of the total homology of $a$.} $\supph(a) \subseteq X$ of a perfect complex $a \in \Der_{\qc}(X)^c$ corresponds to the closed subset $\supp(a) \subseteq \Spc(\Der_{\qc}(X)^c)$.

For the translation to tt-geometry in this form, see \cite[Corollary 5.6 and Theorem 6.3]{Balmer05a} and \cite[Theorem 8.5]{BuanKrauseSolberg07}.
\end{Exa}

\begin{Rem}
When $X = \Spec(R)$ is affine, there is a simple description of the inverse to the above isomorphism, given by the following general construction and theorem of Balmer \cite[Theorem 5.3]{Balmer10b}.
\end{Rem}

\begin{Thm}[Balmer]\label{thm:comparison-map}
Let $\cat K \in \tring$ be an essentially small tt-category. Then there is a natural morphism of locally Dirac ringed spaces
\[
    \rho \colon \Spc(\cat K) \to \Spec^h(\pi_*\End_{\cat K}(\unit))
\]
whose underlying map is the continuous, inclusion-reversing assignment
\[
    \rho(\cat P) = \SET{f\in\pi_*\End_{\cat K}(\unit)}{\cone(f) \notin \cat P}.
\]
\end{Thm}

\begin{Rem}
It is not always the case that $\Spc(\cat K)$ is a scheme; for instance, taking $\cat K = \SH^c$ gives a counterexample. See \cite[Proposition 9.7]{Balmer10b}.
\end{Rem}
\section{Completions of tt-categories}\label{sec:completion}
The tensor-triangular analog of $I$-adic completion is completion with respect to a closed subset of the Balmer spectrum. In this short section we recall the notion, and review the most basic example of the derived category of a commutative ring. Nothing in this section is new; the ideas date back to at least Greenlees \cite{Greenlees01}, and we follow the recent exposition in \cite{balmer2025tateintermediatevaluetheorem}.

\begin{Def}
Let $\cat C$ be a rigidly compactly generated stable homotopy theory, and let $Y \subseteq \Spc(\cat C^d)$ be a Thomason subset. We set
\[
  \cat C^d_Y \coloneqq \SET{x\in\cat C^d}{\supp(x)\subseteq Y}
  \quad \text{ and } \quad 
  \cat C_Y \coloneqq \Loc\langle \cat C^d_Y\rangle.
\]
\end{Def}

\begin{Rem}
The inclusion $\cat C_Y \hookrightarrow \cat C$ admits a right adjoint given by tensoring with the Balmer--Favi idempotent $e_Y$ \cite{BalmerFavi11}. Our primary interest, however, lies in the double right orthogonal to $\cat C_Y$. Recall that if $\cat E\subseteq \cat C$ is any class of objects, we write
\[
  \cat E^{\perp} \coloneqq \SET{c\in\cat C}{\Hom_{\cat C}(s,c)=0 \text{ for all } s\in\cat E}.
\]
Note that this is closed under suspension in $\cat C$.
\end{Rem}

\begin{Def}\label{def:completion}
    The \emph{completion} of $\cat C$ at $Y \subseteq \Spc(\cat C^d)$ is the full subcategory
    \[
    \widehat{\cat C}_Y \coloneqq (\cat C_Y)^{\perp}{}^{\perp}.
    \]
    Objects of $\widehat{\cat C}_Y $ are called \emph{$Y$-complete}.
\end{Def}

\begin{Rem}\label{rem:completion-formula}
 The inclusion $\widehat{\cat C}_Y \hookrightarrow \cat C$ has a symmetric monoidal left adjoint, given by $\widehat{(-)}_Y \coloneqq \ihom{e_Y,-}$. The category $\widehat{\cat C}_Y$ is a tensor-triangulated category under the localized tensor product. It is not, however, rigidly-compactly generated, except in trivial cases (see \cite[Remark~2.12]{balmer2024perfectcomplexescompletion}). Note that via the inclusion we can also consider $\widehat{(-)}_Y$ as an endofunctor on $\cat C$.
\end{Rem}

\begin{Rem}\label{rem:completion-as-bousfield}
    We are often interested in the case where $Y = \supp(A)$ for some $A \in \cat C^d$. In this case, the completion functor $\widehat{(-)}_{\supp(A)}$ coincides with Bousfield localization with respect to $A$ (see, for example, \cite[Proposition~2.34]{bhv1}).
\end{Rem}

\begin{Exa}\label{exa:derived-complete-ring}
The most basic example is as follows. Let $R$ be a noetherian ring, $I\subseteq R$ an ideal, and consider the closed subset $Y=V(I)\subseteq \Spec(R)\cong \Spc(\Der(R)^c)$ (see \Cref{exa:balmer-scheme}). Then the functor
\[
\widehat{(-)}_I \colon \Der(R)\ \longrightarrow\ 
\widehat{\Der(R)}_{I}\coloneqq \widehat{\Der(R)}_{V(I)}
\]
is the derived $I$-adic completion functor; see, for instance, \cite{GreenleesMay92,DwyerGreenlees02}. Note that even if $M$ is a discrete $R$-module, the derived completion $M^{\wedge}_I$ need not agree with the classical $I$-adic completion. For example, with $R=\mathbb{Z}$, $I=(p)$, and $M=\bigoplus_{n\ge1}\mathbb{Z}/p^n$, the two notions differ \cite[Example~10.1.16]{MayPonto2012More}.
\end{Exa}
\section{The formal spectrum of a tt-category}\label{sec:formal-spectrum}
We now turn to the definition of the formal spectrum in the tensor-triangular setting.

\begin{Def}\label{def:formal-balmer}
Let $\cat K \in \tring$ be a 2-ring and let $\cat J=\cat K_Y$ be the thick $\otimes$-ideal in $\cat K$ corresponding to a Thomason subset $Y\subseteq \Spc(\cat K)$ (note that every thick $\otimes$-ideal is of this form by \Cref{thm:balmer}). Write $\cat C$ for the stable homotopy theory associated to $\cat K$ (\Cref{Rem:ind-completion}). The \emph{formal spectrum} of $\cat K$ (with respect to $Y$), denoted $\Spf(\cat K,Y)$, is the Dirac ringed space
\[
(|\Spf(\cat K,Y)|, \mathcal{O}_{\Spf(\cat K,Y)})
\]
defined as follows:
\begin{enumerate}
    \item The underlying topological space is
    \[
    |\Spf(\cat K,Y)| \coloneqq Y \subseteq \Spc(\cat K)
    \]
    equipped with the subspace topology. A basis of open sets is given by
    \[
    \overline{U}(c) \coloneqq U(c)\cap Y = \SET{\cat P\in Y}{c\in \cat P}.
    \]

    \item Let ${}_{\mathrm p}\mathcal O_{\cat K,Y}$ denote the presheaf on this basis with values in $\mathrm{gr}\!\CAlg^{\heartsuit}$ defined by
    \[
      {}_{\mathrm p}\mathcal O_{\cat K,Y}\bigl(\overline U(c)\bigr)
      \coloneqq \pi_*\Hom_{\cat C}\Bigl(\unit, \widehat{(\unit[c^{-1}])}_Y\Bigr),
    \]
    and define $\mathcal O_{\Spf(\cat K,Y)}\coloneqq a({}_{\mathrm p}\mathcal O_{\cat K,Y})$. Here $\unit[c^{-1}]$ denotes the tensor unit of the localization $\cat C[c^{-1}] = \cat C(U(c))$, viewed as an object of $\cat C$ via the right adjoint to the localization functor, and $\widehat{(-)}_Y$ denotes $Y$-completion as introduced in \Cref{sec:completion}. We verify in \Cref{lem:presheaf-well-defined} that this assignment depends only on the basic open $\overline U(c)$, and hence defines a presheaf on the displayed basis of $Y$.
\end{enumerate}
\end{Def}
\begin{Rem}\label{rem:adjoint-sheaf}
    Recall that $\widehat{(-)}_Y \simeq \ihom{e_Y,-}$. Therefore, the defining presheaf can equivalently be written as
    \[
      {}_{\mathrm p}\mathcal O_{\cat K,Y}\bigl(\overline U(c)\bigr)
      \cong \pi_*\Hom_{\cat C}\bigl(e_Y,\unit[c^{-1}]\bigr).
    \]
\end{Rem}
\begin{Lem}\label{lem:completed-units-well-defined}
If $U_1\subseteq U_2$ are quasi-compact opens of $\Spc(\cat K)$ and $U_1\cap Y=U_2\cap Y$, then the canonical map
\[
  \widehat{(\unit_{U_2})}_Y\longrightarrow \widehat{(\unit_{U_1})}_Y
\]
is an equivalence in $\CAlg(\cat C)$.
\end{Lem}
\begin{proof}
    Write $Z_i=\Spc(\cat K)\setminus U_i$. Then $Z_2\subseteq Z_1$ and $\unit_{U_i}\simeq f_{Z_i}$. The fiber of the canonical map $f_{Z_2}\to f_{Z_1}$ is $e_{Z_1}\otimes f_{Z_2}$. Since $U_1\cap Y=U_2\cap Y$, we have $Y\cap Z_1\subseteq Z_2$, and hence
\[
e_Y\otimes e_{Z_1}\otimes f_{Z_2}
\simeq e_{Y\cap Z_1}\otimes f_{Z_2}
\simeq 0
\]
by \cite[Lemma 1.27]{BarthelHeardSers2023Stratification}. Applying $\ihom{e_Y,-}$ to the fiber sequence and using the adjunction equivalence
\[
  \ihom{e_Y,e_Y\otimes W}\xrightarrow{\sim} \ihom{e_Y,W}
\]
shows that $\widehat{f_{Z_2}}_Y\to\widehat{f_{Z_1}}_Y$ is an equivalence. Finally, this is an equivalence in $\CAlg(\cat C)$ because $\widehat{(-)}_Y=iL$ is lax symmetric monoidal: $L$ is symmetric monoidal and its fully faithful right adjoint $i$ is lax symmetric monoidal.
\end{proof}
\begin{Lem}\label{lem:presheaf-well-defined}
Let $\cat B$ be the basis $\{U(c)\}_{c\in\cat K}$ of $\Spc(\cat K)$ and let $\overline{\cat B}$ be the basis $\{\overline U(c)\}_{c\in\cat K}$ of $Y$.
\begin{enumerate}[label=\textup{(\alph*)}]
\item The assignment $U(c)\mapsto\pi_*\Hom_{\cat C}(e_Y,\unit[c^{-1}])$ defines a presheaf $\cat B^{\op}\to\mathrm{gr}\!\CAlg^{\heartsuit}$.
\item This presheaf descends along $U\mapsto U\cap Y$ to the presheaf ${}_{\mathrm p}\mathcal O_{\cat K,Y}$ on $\overline{\cat B}$. In particular, the presheaf in \Cref{def:formal-balmer} is well-defined.
\end{enumerate}
\end{Lem}

\begin{proof}
For (a), the open $U(c)$ determines its complement $\supp(c)$, and $\unit[c^{-1}]\simeq f_{\supp(c)}$ is the corresponding Balmer--Favi idempotent, so it is independent of the choice of $c$ presenting $U(c)$. If $U(c_1)\subseteq U(c_2)$, then $\supp(c_2)\subseteq\supp(c_1)$ and the canonical map $f_{\supp(c_2)}\to f_{\supp(c_1)}$ induces the restriction map. Functoriality of these canonical maps gives identities and composition.

For (b), \Cref{lem:completed-units-well-defined} shows that if $U(c_1)\subseteq U(c_2)$ and $\overline U(c_1)=\overline U(c_2)$, then restriction induces an isomorphism
\[
  \pi_*\Hom_{\cat C}(e_Y,\unit[c_2^{-1}])
  \xrightarrow{\cong}
  \pi_*\Hom_{\cat C}(e_Y,\unit[c_1^{-1}]).
\]
For arbitrary representatives of the same basic open, compare both with $U(c_1\oplus c_2)=U(c_1)\cap U(c_2)$. More generally, if $\overline U(c_1)\subseteq\overline U(c_2)$, define restriction by the canonical zigzag
\[
  \pi_*\Hom_{\cat C}(e_Y,\unit[c_2^{-1}])
  \longrightarrow
  \pi_*\Hom_{\cat C}(e_Y,\unit[(c_1\oplus c_2)^{-1}])
  \xleftarrow{\ \cong\ }
  \pi_*\Hom_{\cat C}(e_Y,\unit[c_1^{-1}]).
\]
The isomorphism just proved shows that this is independent of the representatives, and compatibility with composition follows from the restriction maps after passing to common intersections. Thus the assignment descends to the poset $\overline{\cat B}$.
\end{proof}
The previous result shows that $\Spf(\cat K,Y)$ is a well defined Dirac ringed space. What we have not yet shown is that the stalks are Dirac local. Therefore we ask:
\begin{Que}
    Is $\Spf(\cat K,Y)$ always a \emph{locally} Dirac ringed space?
\end{Que}

   \begin{Rem}\label{rem:completed-sheaf}
When $\Spc(\cat K)$ is noetherian, \Cref{sec:categorical-sheaf} constructs a structure sheaf valued in $\CAlg(\Pr\nolimits^L_{\mathrm{st}})$ on $Y$. Its value on $U\cap Y$ is naturally equivalent to the category $\widehat{\cat C(U)}_{Y\cap U}$ of complete objects in $\cat C(U)$; see \Cref{thm:categorical-formal-sheaf}. Taking endomorphisms of the tensor units, applying $\pi_*$ sectionwise, and sheafifying recovers $\mathcal O_{\Spf(\cat K,Y)}$ by \Cref{cor:categorical-formal-structure-sheaf}. We expect that the noetherian assumption can be removed, but do not pursue this here.
\end{Rem}

\begin{Exa}\label{exa:whole-spectrum}
Taking $\cat J=\thick_{\otimes}\langle \unit\rangle$ (equivalently $Y=\Spc(\cat K)$), the construction reduces to the usual Balmer spectrum:
\[
\Spf(\cat K,\Spc(\cat K)) \cong \Spc(\cat K)
\]
as Dirac ringed spaces (compare with \eqref{eq:complicated-balmer-sheaf}).
\end{Exa}

\begin{Rem}
There is a subtle difference from the algebraic situation: for rings, one typically forms formal spectra only along \emph{closed} subsets of $\Spec(R)$, whereas here we allow arbitrary Thomason subsets $Y\subseteq \Spc(\cat K)$. These two classes of subsets are not comparable in general: a Thomason subset need not be closed, and a closed subset need not be Thomason. When $\Spc(\cat K)$ is noetherian, however, every closed subset is Thomason, and the Thomason subsets are precisely the specialization-closed subsets.
\end{Rem}

The formal spectrum is functorial in the obvious way:

\begin{Lem}\label{lem:functoriality}
Let $\cat K,\cat L$ be 2-rings and let $Y\subseteq \Spc(\cat K)$ and $Y'\subseteq \Spc(\cat L)$ be Thomason subsets. Suppose $F\colon \cat K\to \cat L$ is an exact symmetric monoidal functor and write $\varphi\coloneqq \Spc(F)\colon \Spc(\cat L)\to \Spc(\cat K)$ for the induced map. If $\varphi(Y')\subseteq Y$, then there is a morphism of Dirac ringed spaces
\[
  (\varphi,\varphi^\#)\colon \Spf(\cat L,Y')\longrightarrow \Spf(\cat K,Y).
\]
\end{Lem}

\begin{proof}
On underlying spaces this is the restriction $\varphi|_{Y'}\colon Y'\to Y$. For the sheaf map, we set $\cat C\coloneqq \Ind(\cat K)$ and $\cat D\coloneqq \Ind(\cat L)$; the functor $F$ extends to a colimit-preserving symmetric monoidal functor $F\colon \cat C\to \cat D$.

Let $\overline U(c)\coloneqq U(c)\cap Y$ be a basic open of $\Spf(\cat K,Y)$, with $c\in \cat K$ compact. Using \Cref{rem:adjoint-sheaf} we have
\[
  {}_{\mathrm p}\mathcal O_{\cat K,Y}(\overline U(c))
  \cong \pi_*\Hom_{\cat C}(e_Y,\unit [c^{-1}]).
\]
Define on the basis the composite
\[
\begin{split}
\pi_*\Hom_{\cat C}(e_Y,\unit [c^{-1}])
&\longrightarrow
\pi_*\Hom_{\cat D}(F(e_Y),F(\unit [c^{-1}]))
\\
&\xlongrightarrow{\sim}\pi_*\Hom_{\cat D}(e_{\varphi^{-1}(Y)},\unit[F(c)^{-1}])
\\
&\xlongrightarrow{\mathrm{res}}\pi_*\Hom_{\cat D}(e_{Y'},\unit[F(c)^{-1}]),
\end{split}
\]
where $F(e_Y)\simeq e_{\varphi^{-1}(Y)}$ and $F(\unit[c^{-1}])\simeq \unit[F(c)^{-1}]$ by the functoriality of idempotents (see \cite[Proposition~5.11]{BalmerSanders17}), and the last map is induced by the morphism of idempotents $e_{Y'}\to e_{\varphi^{-1}(Y)}$ coming from $Y'\subseteq \varphi^{-1}(Y)$. Note that the target is precisely ${}_{\mathrm p}\mathcal O_{\cat L,Y'}(\overline U(F(c)))$, where $\overline U(F(c)) \coloneqq U(F(c)) \cap Y'$ and $\overline U(F(c))=\varphi^{-1}(\overline U(c))$ as open sets of $Y'$. The constructions are all natural in $c$ and compatible with restriction maps, and hence define a morphism of presheaves. Sheafifying gives a morphism
\[
  \varphi^\#\colon \mathcal O_{\Spf(\cat K,Y)}\longrightarrow \varphi_*\mathcal O_{\Spf(\cat L,Y')}.
\]
Therefore, $(\varphi,\varphi^\#)$ is a morphism of Dirac ringed spaces $\Spf(\cat L,Y')\to \Spf(\cat K,Y)$.
\end{proof}

\begin{Exa}
As a special case of \Cref{lem:functoriality}, take $F = \id$ and $Y = \Spc(\cat K)$. Together with \Cref{exa:whole-spectrum}, this shows that for any Thomason subset $Y' \subseteq \Spc(\cat K)$, there is a canonical morphism of Dirac ringed spaces
\[
  \Spf(\cat K,Y') \longrightarrow \Spc(\cat K),
\]
which realizes the inclusion $Y' \hookrightarrow \Spc(\cat K)$ on underlying spaces. Here the morphism of sheaves is induced by the map $e_{Y'} \to \unit$.
\end{Exa}
\begin{Not}\label{not:completion-theorem}
Let $\cat K$ be a $2$-ring, let $\cat C=\Ind(\cat K)$, and fix a Thomason subset $Y\subseteq\Spc(\cat K)$. Set
\[
  L\coloneqq\widehat{(-)}_Y\colon\cat C\longrightarrow\widehat{\cat C}_Y,\qquad
  \cat L \coloneqq \widehat{\cat C}^{\dual}_Y,\qquad
  \cat D \coloneqq \Ind(\cat L).
\]
Note that completion induces a symmetric monoidal functor $L|_{\cat K}\colon\cat K\to\cat L$. Set
\[
  \varphi \coloneqq \Spc(L|_{\cat K}),\qquad Y_1 \coloneqq \varphi^{-1}(Y),
\]
and let $\Lambda=\Ind(L|_{\cat K})\colon\cat C\to\cat D$ and $\Psi\colon\cat D\to\widehat{\cat C}_Y$ be the cocontinuous symmetric monoidal functor extending the inclusion $\cat L\hookrightarrow\widehat{\cat C}_Y$. The cocontinuous functors $\Psi\Lambda$ and $L$ agree on $\cat K$, so the universal property of $\Ind(\cat K)$ gives a natural equivalence $\Psi\Lambda\simeq L$. By \Cref{lem:functoriality}, completion also induces
\begin{equation}\label{eq:completion-dual}
  (\varphi,\varphi^\#)\colon
  \Spf(\cat L,Y_1)\longrightarrow\Spf(\cat K,Y).
\end{equation}
With this setup, our ultimate goal is to show that \eqref{eq:completion-dual} is an isomorphism. But first, we need some preliminary results.
\end{Not}
\begin{Thm}[Balmer--Sanders]\label{lem:completion-spectrum}
In the notation of \eqref{not:completion-theorem}, the restriction $\varphi|_{Y_1}\colon Y_1\xrightarrow{\sim}Y$ is a homeomorphism, and
\[
  \cat L_{Y_1}\simeq (\widehat{\cat C}_Y)^c \simeq \cat K_Y.
\]
\end{Thm}

\begin{proof}
On homotopy categories, this is \cite[Theorem~4.6]{balmer2025tateintermediatevaluetheorem}. The equivalence of homotopy categories is induced by an underlying exact functor between the corresponding stable $\infty$-categories. Since an exact functor of stable $\infty$-categories that induces an equivalence on homotopy categories is itself an equivalence, the result upgrades to the asserted equivalence of stable $\infty$-categories.
\end{proof}
\begin{Lem}\label{lem:completion-idempotent}
For every $M\in\cat D$, the functor $\Psi$ sends the $Y_1$-completion map $M\to\widehat M_{Y_1}$ to an equivalence. It therefore factors through $Y_1$-completion and induces an equivalence
\[
  \overline\Psi\colon\widehat{\cat D}_{Y_1}\xrightarrow{\ \sim\ }\widehat{\cat C}_Y.
\]
Consequently, if $\Psi^R$ is the right adjoint of $\Psi$, then the adjunction unit identifies $\Psi^R\Psi$ with $\widehat{(-)}_{Y_1}$.
\end{Lem}

\begin{proof}
By construction, $\Psi\colon\cat D=\Ind(\cat L)\to\widehat{\cat C}_Y$ is the filtered-colimit-preserving extension of the fully faithful inclusion $\cat L\hookrightarrow\widehat{\cat C}_Y$. For $k\in\cat K_Y$, regarded as a compact object of $\widehat{\cat C}_Y$ via \Cref{lem:completion-spectrum}, and $M\in\cat D$, we claim there is a natural equivalence of mapping spectra
\[
  \Hom_{\cat D}(k,M)\longrightarrow\Hom_{\widehat{\cat C}_Y}(k,\Psi M).
\]
Indeed, as $M$ varies, the displayed map is a natural transformation of functors $\cat D\to\Sp$. Both functors preserve filtered colimits. The transformation is an equivalence on $\cat L$, where $\Psi$ restricts to the inclusion $\cat L\hookrightarrow\widehat{\cat C}_Y$, and therefore on $\Ind(\cat L)$ by \cite[Proposition~5.3.5.10]{HTTLurie}.

By \Cref{lem:completion-spectrum}, $\cat L_{Y_1}$ identifies with $\cat K_Y$ under the functor $\Psi$. If $N$ belongs to the kernel $(\cat L_{Y_1})^\perp$ of $Y_1$-completion, the displayed equivalence shows that $\Hom_{\widehat{\cat C}_Y}(k,\Psi N)\simeq0$ for every $k\in\cat K_Y$. Since $\cat K_Y$ generates $\widehat{\cat C}_Y$, we have $\Psi N\simeq0$. Thus $\Psi$ factors through a functor $\overline\Psi\colon\widehat{\cat D}_{Y_1}\to\widehat{\cat C}_Y$.

The source and target are compactly generated, with compact objects $\cat L_{Y_1}$ and $\cat K_Y$, respectively \cite[Remark~3.13]{balmer2025tateintermediatevaluetheorem}; hence they identify with $\Ind(\cat L_{Y_1})$ and $\Ind(\cat K_Y)$. Under the identification of \Cref{lem:completion-spectrum}, $\overline\Psi$ restricts to the identity on compact objects. Therefore, \cite[Proposition~5.3.5.11]{HTTLurie} shows that $\overline\Psi$ is an equivalence. In particular, $\Psi$ sends every $Y_1$-completion map to an equivalence. If $L_1\dashv\iota_1$ denotes $Y_1$-completion, then $\Psi\simeq\overline\Psi L_1$ and $\Psi^R\simeq\iota_1\overline\Psi^{-1}$, giving $\Psi^R\Psi\simeq\iota_1L_1=\widehat{(-)}_{Y_1}$, compatibly with the units.
\end{proof}
\begin{Thm}\label{thm:completion-theorem}
Let $\cat K$ be a 2-ring and let $Y\subseteq \Spc(\cat K)$ be a Thomason subset. Then the canonical morphism
\[
    (\varphi,\varphi^\#)\colon 
    \Spf(\widehat{\cat C}^{\dual}_Y,\varphi^{-1}(Y)) 
    \longrightarrow 
    \Spf(\cat K,Y)
\]
is an isomorphism of Dirac ringed spaces.
\end{Thm}

\begin{proof}
The underlying map is a homeomorphism by \Cref{lem:completion-spectrum}. It remains to compare the defining presheaves on a basic open $\overline U(c)$. The inverse images of these opens form a basis of $Y_1$ because $\varphi|_{Y_1}$ is a homeomorphism. Put $X=\unit[c^{-1}]$. By functoriality of Balmer--Favi idempotents and localization,
\[
  e_{Y_1}\simeq\Lambda(e_Y),\qquad
  \unit_1[L(c)^{-1}]\simeq\Lambda(X),
\]
and the map of presheaves is obtained by applying $\pi_*$ to the following map of mapping spectra induced by $\Lambda$:
\begin{equation}\label{eq:completion-sections}
  \Hom_{\cat C}(e_Y,X)\longrightarrow
  \Hom_{\cat D}(\Lambda e_Y,\Lambda X).
\end{equation}

Applying $\Psi$ to the target of \eqref{eq:completion-sections} gives an equivalence
\[
  \Hom_{\cat D}(\Lambda e_Y,\Lambda X)
  \xrightarrow{\sim}
  \Hom_{\widehat{\cat C}_Y}(L e_Y,L X).
\]
Indeed, under $\Psi\dashv\Psi^R$ this map is induced by the $Y_1$-completion $\id\to\Psi^R\Psi$ from \Cref{lem:completion-idempotent}, and mapping out of $e_{Y_1}$ inverts that map since $e_{Y_1}\otimes e_{Y_1}\simeq e_{Y_1}$. The composite with \eqref{eq:completion-sections} is the map induced by $L\simeq\Psi\Lambda$. Under the adjunction $L\dashv\iota$ of \Cref{rem:completion-formula}, it identifies with the map
\[
\Hom_{\cat C}(e_Y,X)\longrightarrow\Hom_{\cat C}(e_Y,\ihom{e_Y,X})
\]
given by postcomposition with the completion unit $X\to\ihom{e_Y,X}$. Under the tensor--hom adjunction this is precomposition with the equivalence $e_Y\otimes e_Y\to e_Y$ induced by $e_Y\to\unit$. It is therefore an equivalence because $e_Y$ is idempotent.
Thus, in the composite from the source of \eqref{eq:completion-sections} to $\Hom_{\widehat{\cat C}_Y}(L e_Y,L X)$, the second map and the composite are equivalences. Hence \eqref{eq:completion-sections} is an equivalence. Applying $\pi_*$ on every basic open and sheafifying proves that $\varphi^\#$ is an isomorphism.
\end{proof}

\begin{Rem}
    Because of this theorem, we will usually write $\Spf(\widehat{\cat C}^{\dual}_Y)$ instead of the more cumbersome $\Spf(\widehat{\cat C}^{\dual}_Y,\varphi^{-1}(Y))$.
\end{Rem}

We finish this section with two locality results.
\begin{Lem}\label{lem:open-restriction-transitivity}
Let $W\subseteq U\subseteq\Spc(\cat K)$ be quasi-compact opens, and let $W'\subseteq\Spc(\cat K(U))\simeq U$ denote the open corresponding to $W$. Then there is a canonical symmetric monoidal equivalence
\[
  \cat K(U)(W')\simeq \cat K(W)
\]
under $\cat K$. These equivalences are transitive for chains of quasi-compact opens.
\end{Lem}

\begin{proof}
Let $Z_U$ and $Z_W$ be the closed complements of $U$ and $W$. Since $W\subseteq U$, we have $Z_U\subseteq Z_W$. On Ind-categories, localization at $U$ is tensoring with $f_{Z_U}$, and then localizing further at $W'$ is tensoring with the corresponding idempotent, so the composite localization is tensoring with
\[
  f_{Z_U}\otimes f_{Z_W}\simeq f_{Z_U\cup Z_W}\simeq f_{Z_W}
\]
by \cite[Lemma 1.27]{BarthelHeardSers2023Stratification}. This is the localization directly associated to $W$. Passing to compact objects and idempotent completions gives the stated equivalence. The same idempotent calculation proves transitivity.
\end{proof}
\begin{Thm}\label{Thm:spf-open-restriction}
Let $\cat K$ be a 2-ring and $Y\subseteq \Spc(\cat K)$ a Thomason subset. Let $U \subseteq \Spc(\cat K)$ be a quasi-compact open subset, and let $q_U \colon \cat K\to \cat K(U)$ be the functor from \Cref{rem:verdier-quotient}. Set $Y'\coloneqq Y\cap U\subseteq \Spc(\cat K(U))$. Then the restriction of $\Spf(\cat K,Y)$ to $Y'$ identifies with $\Spf(\cat K(U),Y')$ as a Dirac ringed space:
\[
\Spf(\cat K(U),Y') \xrightarrow{\cong} \Spf(\cat K,Y)\big|_{Y'}.
\]
\end{Thm}

\begin{proof}
Since $U$ is quasi-compact, there exists $d \in \cat K$ such that $U = U(d)$ \cite[Proposition~2.14]{Balmer05a}. Consequently, the localization $\cat K(U)$ is identified with the idempotent completion of the Verdier quotient, $(\cat K[d^{-1}])^{\natural}$.

By functoriality (\Cref{lem:functoriality}) applied to $q_U$, we obtain a morphism
\[
(\varphi,\varphi^\#)\colon \Spf(\cat K(U),Y')\longrightarrow \Spf(\cat K,Y).
\]
On underlying topological spaces, $\varphi$ is the inclusion $Y' \hookrightarrow Y$. Since the image of $\varphi$ lies in the open set $Y'$, this morphism factors through the restriction $\Spf(\cat K,Y)|_{Y'}$. To prove that this factorization is an isomorphism, it remains to show that the induced map on structure sheaves is an isomorphism. It suffices to verify this on a basis of the open sets of $Y'$.

Let $\overline{V} \subseteq Y'$ be a basic open set. We can write $\overline{V} = U(c) \cap Y'$ for some $c \in \cat K$. By replacing $c$ with $c \oplus d$ (which satisfies $U(c \oplus d) = U(c) \cap U(d)$), we may assume without loss of generality that $U(c) \subseteq U(d)=U$.

Let $\cat C \coloneqq \Ind(\cat K)$ and let $L \colon \cat C \to \cat C(U)$ be the localization functor, with fully faithful right adjoint $\iota$. We compare the defining presheaves over $\overline{V}$:
\begin{enumerate}
   \item On $\Spf(\cat K, Y)$: Since $U(c) \cap Y = \overline{V}$ (as $U(c) \subseteq U$), the defining presheaf value is
    \[
    {}_{\mathrm p}\mathcal O_{\cat K,Y}(\overline V)
    \cong \pi_* \Hom_{\cat C}(e_Y, \unit_{\cat C}[c^{-1}]).
    \]
    \item On $\Spf(\cat K(U), Y')$: The defining presheaf value is
    \[
    {}_{\mathrm p}\mathcal O_{\cat K(U),Y'}(\overline V)
    \cong \pi_* \Hom_{\cat C(U)}(e_{Y'}, \unit_{\cat C(U)}[c^{-1}]).
    \]
\end{enumerate}
Since $U(c) \subseteq U(d)$, the object $\unit_{\cat C}[c^{-1}]$ is $U$-local. Thus, the canonical map $\unit_{\cat C}[c^{-1}] \to \iota(\unit_{\cat C(U)}[c^{-1}])$ is an equivalence. Using the adjunction $L \dashv \iota$ and the fact that $L(e_Y) \simeq e_{Y \cap U} = e_{Y'}$ (see \cite[Proposition~5.11]{BalmerSanders17}), we obtain natural isomorphisms:
\[
\Hom_{\cat C}(e_Y, \unit_{\cat C}[c^{-1}]) 
\cong \Hom_{\cat C(U)}(L(e_Y), \unit_{\cat C(U)}[c^{-1}])
\cong \Hom_{\cat C(U)}(e_{Y'}, \unit_{\cat C(U)}[c^{-1}]).
\]
This shows that the map of defining presheaves is an isomorphism on basic open sets; after sheafification, the map of sheaves is an isomorphism.
\end{proof}
\section{A sheaf of categories}\label{sec:categorical-sheaf}
We now move on to lifting the sheaf constructed in the previous section to the level of $\infty$-categories. We makes the following assumption.
\begin{Ass}
    Assume throughout this section that $\Spc(\cat K)$ is noetherian, so that every open subset is quasi-compact. We expect that this assumption is unnecessary, but it simplifies the arguments and holds in many examples of interest.
\end{Ass}
\begin{Def}
Let $\cat K\in\tring$, set $\cat C=\Ind(\cat K)$, and fix a Thomason subset $Y\subseteq\Spc(\cat K)$. For an arbitrary open subset $V\subseteq\Spc(\cat K)$, set
\[
\mathcal F_Y(V)\coloneqq\widehat{\cat C}_Y\otimes_{\cat C}\mathcal O_{\cat K}(V)\in\CAlg(\Pr\nolimits_{\mathrm{st}}^L),
\]
where $\mathcal O_{\cat K}$ is the sheaf constructed in \Cref{remark:sheaf-of-categories}. Here $\otimes_{\cat C}$ denotes the relative Lurie tensor product in $\Pr\nolimits_{\mathrm{st}}^L$, equipped with the commutative algebra structure induced by the symmetric monoidal completion functor $\cat C\to\widehat{\cat C}_Y$ and the restriction functor $\cat C\simeq\mathcal O_{\cat K}(\Spc(\cat K))\to\mathcal O_{\cat K}(V)$.
\end{Def}
\begin{Prop}\label{prop:completed-category-sheaf}
The assignment $\mathcal F_Y$ defines a $\CAlg(\Pr\nolimits_{\mathrm{st}}^L)$-valued sheaf on $\Spc(\cat K)$.
\end{Prop}
\begin{proof}
    By \cite[Theorem~5.28(a)]{aoki2025higherzariskigeometry}, the assignment $V \mapsto \widehat{\cat C}_Y \otimes_{\cat C} \mathcal O_{\cat K}(V)$ on quasi-compact opens extends to a sheaf of $\cat C$-modules in $\Pr\nolimits_{\mathrm{st}}^L$. Since every open subset is quasi-compact, this is precisely the underlying module-valued functor of $\mathcal F_Y$. The forgetful functors from commutative algebra objects and from $\cat C$-modules create limits, so $\mathcal F_Y$ is a sheaf with the asserted values.
\end{proof}
We will also need the following result.
\begin{Lem}\label{Lem:sheaf-identification}
With notation as above, let $V \subseteq \Spc(\cat K)$ be open. There is a natural symmetric monoidal equivalence
\[
\mathcal F_Y(V) \simeq \widehat{\cat C(V)}_{Y \cap V}.
\]
\end{Lem}
\begin{proof}
Let $Z$ denote the closed complement of $V$. Since $V$ is quasi-compact, $\cat C(V)$ is a finite localization of $\cat C$, and there are symmetric monoidal equivalences $\mathcal O_{\cat K}(V) \simeq \cat C(V) \simeq \Mod_{f_Z}(\cat C)$. Base change for module categories \cite[Theorem~4.8.4.6 and the proof of Proposition~4.8.5.1]{HALurie} gives an equivalence
\[
\mathcal F_Y(V) \simeq \widehat{\cat C}_Y \otimes_{\cat C} \Mod_{f_Z}(\cat C) \simeq \Mod_{L f_Z}(\widehat{\cat C}_Y),
\]
where $L \colon \cat C \to \widehat{\cat C}_Y$ is the completion functor. Since $L$ is symmetric monoidal and $f_Z$ is commutative, the canonical comparison is symmetric monoidal, and hence the displayed equivalence is an equivalence in $\CAlg(\Pr\nolimits_{\mathrm{st}}^L)$.
It remains to identify this with the category of $(Y \cap V)$-complete objects in $\cat C(V)$. Consider the smashing localization adjunction
\[
F = f_Z \otimes - \colon \cat C \rightleftarrows{}{} \cat C(V) \noloc \iota
\]
and the object $e_Y \in \cat C$. The right adjoint $\iota$ is conservative and preserves colimits, and the adjunction satisfies the projection formula. Since Bousfield localization at $e_Y$ is $Y$-completion, \cite[Lemma~3.7]{BehrensShah2020Equivariant} gives a symmetric monoidal equivalence
\[
\Mod_{L f_Z}(\widehat{\cat C}_Y) \simeq \bigl(\Mod_{f_Z}(\cat C)\bigr)_{f_Z \otimes e_Y},
\]
where the right-hand side denotes Bousfield localization at $f_Z \otimes e_Y$. For an object $M \in \Mod_{f_Z}(\cat C)$, we have
\[
(f_Z \otimes e_Y) \otimes_{f_Z} M \simeq e_Y \otimes M,
\]
so $M$ is acyclic precisely when $e_Y \otimes M \simeq 0$. Since $f_Z \otimes e_Y$ is the torsion idempotent associated to $Y \cap V$ in $\cat C(V)$ by \cite[Proposition~5.11]{BalmerSanders17}, these acyclic objects form $\cat C(V)_{Y \cap V}^{\perp}$. Their right orthogonal is, by definition, $\widehat{\cat C(V)}_{Y \cap V}$, proving the result. The identifications are natural with respect to restriction to smaller opens. Explicitly, for $W\subseteq V$, the induced restriction functor on complete objects is
\[
\widehat{\cat C(V)}_{Y\cap V}\longrightarrow\widehat{\cat C(W)}_{Y\cap W},
\qquad M\longmapsto\widehat{(M|_W)}_{Y\cap W},
\]
where $M|_W$ denotes restriction in the uncompleted categories. This is the functor induced on Bousfield localizations by $\cat C(V)\to\cat C(W)$, as in the cited lemma.
\end{proof}
In order to construct the sheaf on $Y \subseteq \Spc(\cat K)$ we will need the following topological result. 
\begin{Lem}\label{Lem:subspace-gluing}
Let $i \colon Y \hookrightarrow X$ be the inclusion of a subspace and let $\cat E$ be an $\infty$-category with all small limits. Then the pushforward functor
\[
i_* \colon \Shv(Y;\cat E) \longrightarrow \Shv(X;\cat E)
\]
is fully faithful. Its essential image consists precisely of those sheaves $\cat F$ satisfying
\begin{equation*}\tag{\textasteriskcentered}\label{eq:sheafcondition}
V_1 \subseteq V_2 \subseteq X \text{ open with } V_1 \cap Y = V_2 \cap Y
\quad\Longrightarrow\quad
\cat F(V_2) \xrightarrow{\sim} \cat F(V_1).
\end{equation*}
For such a sheaf $\cat F$, the corresponding sheaf on $Y$ is given by $\cat G(U) \coloneqq \cat F(X \setminus \overline{Y \setminus U})$. In particular, there is a natural equivalence $\cat G(V \cap Y) \simeq \cat F(V)$ for every open subset $V \subseteq X$.
\end{Lem}
\begin{proof}
There is an adjoint pair
\[
q \colon \Open(X) \rightleftarrows{}{} \Open(Y) \noloc r
\]
given by $q(V) = V \cap Y$ and $r(U) = X \setminus \overline{Y \setminus U}$; see \cite[\S1]{ErnePicadoPultr2022Adjoint}. The functor $r$ is fully faithful because $qr(U) = U$, so $q$ is a localization of posets. The pushforward $i_*$ is given by precomposition with $q^{\op}$, so $(i_* \cat G)(V) = \cat G(V \cap Y)$. Since precomposition with a localization is fully faithful, $i_*$ is fully faithful. Moreover, $q$ preserves arbitrary unions and finite intersections, so it carries sheaves on $Y$ to sheaves on $X$. 

Suppose first that $\cat F = i_* \cat G$ for some sheaf $\cat G$ on $Y$. If $V_1 \subseteq V_2$ and $V_1 \cap Y = V_2 \cap Y$, then the restriction map $\cat F(V_2) = \cat G(V_2 \cap Y) \to \cat G(V_1 \cap Y) = \cat F(V_1)$ is an equivalence. Thus every sheaf in the essential image satisfies \eqref{eq:sheafcondition}. Conversely, suppose that $\cat F$ satisfies \eqref{eq:sheafcondition} and define a presheaf $\cat G$ on $Y$ by $\cat G(U) \coloneqq \cat F(r(U))$. By \cite[Theorem~3.19]{Aoki2025Sheaves}, it is enough to verify the following three conditions.
\begin{enumerate}[label=\textup{(\alph*)}]
\item The object $\cat G(\varnothing)$ is terminal.
\item For any opens $U,U' \subseteq Y$, the canonical map $\cat G(U \cup U') \to \cat G(U) \times_{\cat G(U \cap U')} \cat G(U')$ is an equivalence.
\item For every directed subset $D \subseteq \Open(Y)$, the canonical map $\cat G(\bigcup_{U \in D} U) \to \varprojlim_{U \in D} \cat G(U)$ is an equivalence.
\end{enumerate}
For \textup{(a)}, observe that $r(\varnothing) \cap Y = \varnothing$. Hence \eqref{eq:sheafcondition} gives $\cat G(\varnothing) = \cat F(r(\varnothing)) \simeq \cat F(\varnothing)$, which is terminal.  For \textup{(b)}, since $r$ is a right adjoint, we have $r(U \cap U') = r(U) \cap r(U')$. Moreover, $r(U) \cup r(U') \subseteq r(U \cup U')$, and these two opens have the same intersection with $Y$. Condition \eqref{eq:sheafcondition} therefore gives $\cat F(r(U \cup U')) \simeq \cat F(r(U) \cup r(U'))$. The required pullback square for $\cat G$ consequently identifies with the pullback square for $\cat F$ associated to $r(U)$ and $r(U')$.

Finally, to see \textup{(c)}, let $D \subseteq \Open(Y)$ be directed. Since $r$ is order-preserving, the family $\{r(U)\}_{U \in D}$ is directed. Moreover, $\bigcup_{U \in D} r(U) \subseteq r(\bigcup_{U \in D} U)$, and these two opens have the same intersection with $Y$. Thus \eqref{eq:sheafcondition} and the directed-union condition for $\cat F$ give
\[
\cat G\left(\bigcup_{U \in D} U\right)
\simeq \cat F\left(\bigcup_{U \in D} r(U)\right)
\simeq \varprojlim_{U \in D} \cat F(r(U))
= \varprojlim_{U \in D} \cat G(U).
\]
Thus $\cat G$ is a sheaf. Finally, for every open $V \subseteq X$, the inclusion $V \subseteq rq(V)$ becomes an equality after intersecting with $Y$. Hence \eqref{eq:sheafcondition} gives $(i_* \cat G)(V) = \cat F(rq(V)) \xrightarrow{\sim} \cat F(V)$. Therefore $i_* \cat G \simeq \cat F$, and $\cat G(V \cap Y) \simeq \cat F(V)$ naturally in $V$.
\end{proof}
\begin{Thm}\label{thm:categorical-formal-sheaf}
Let $\Spc(\cat K)$ be noetherian. There is a sheaf of stable homotopy theories
\[
\cat O_{\Spf(\cat K,Y)}^{\infty} \colon \Open(Y)^{\op} \longrightarrow \CAlg(\Pr\nolimits_{\mathrm{st}}^L)
\]
satisfying the following properties.
\begin{enumerate}[label=\textup{(\alph*)}]
\item There are natural equivalences $\cat O_{\Spf(\cat K,Y)}^{\infty} \simeq i^{-1} \mathcal F_Y$ and $i_* \cat O_{\Spf(\cat K,Y)}^{\infty} \simeq \mathcal F_Y$, where $i \colon Y \hookrightarrow \Spc(\cat K)$ is the inclusion.
\item On a basic open $\overline U = U \cap Y$, where $U \subseteq \Spc(\cat K)$ is quasi-compact open, there are natural symmetric monoidal equivalences
\[
\cat O_{\Spf(\cat K,Y)}^{\infty}(\overline U) \simeq \widehat{\cat C}_Y \otimes_{\cat C} \mathcal O_{\cat K}(U) \simeq \widehat{\cat C(U)}_{Y \cap U}.
\]
\end{enumerate}
\end{Thm}
\begin{proof}
    By \Cref{prop:completed-category-sheaf}, $\mathcal F_Y$ is a sheaf. We first verify hypothesis \eqref{eq:sheafcondition} of \Cref{Lem:subspace-gluing}. Let $V_1 \subseteq V_2$ be opens with $V_1 \cap Y = V_2 \cap Y$, and write $Z_j = \Spc(\cat K) \setminus V_j$. Both opens are quasi-compact, so \Cref{lem:completed-units-well-defined} implies that the canonical map $L f_{Z_2} \to L f_{Z_1}$ is an equivalence of commutative algebras in $\widehat{\cat C}_Y$, where $L \colon \cat C \to \widehat{\cat C}_Y$ is completion. Indeed, the lemma identifies its image under the fully faithful right adjoint $\widehat{\cat C}_Y \to \cat C$ as an equivalence. Under the natural base-change identifications in the proof of \Cref{Lem:sheaf-identification}, the restriction functor $\mathcal F_Y(V_2) \to \mathcal F_Y(V_1)$ identifies with extension of scalars
\[
\Mod_{L f_{Z_2}}(\widehat{\cat C}_Y) \longrightarrow \Mod_{L f_{Z_1}}(\widehat{\cat C}_Y).
\]
It is therefore an equivalence, proving \eqref{eq:sheafcondition}.

By \Cref{Lem:subspace-gluing}, applied with $\cat E = \CAlg(\Pr\nolimits_{\mathrm{st}}^L)$, there is an essentially unique sheaf $\cat G$ on $Y$ with $i_* \cat G \simeq \mathcal F_Y$, together with natural equivalences $\cat G(V \cap Y) \simeq \mathcal F_Y(V)$ for every open $V \subseteq \Spc(\cat K)$. Set $\cat O^{\infty}_{\Spf(\cat K,Y)} \coloneqq \cat G$. Taking $V = U$ gives the first equivalence in \textup{(b)}, and the second is \Cref{Lem:sheaf-identification}.
For \textup{(a)}, the equivalence $i_* \cat O^{\infty}_{\Spf(\cat K,Y)} \simeq \mathcal F_Y$ is what we have just constructed. Full faithfulness of $i_*$ gives natural equivalences of mapping spaces
\[
\operatorname{Map}(\cat G, \cat H) \simeq \operatorname{Map}(i_* \cat G, i_* \cat H) \simeq \operatorname{Map}(\mathcal F_Y, i_* \cat H)
\]
for every $\CAlg(\Pr\nolimits_{\mathrm{st}}^L)$-valued sheaf $\cat H$ on $Y$. Thus $\cat G$ satisfies the universal property of $i^{-1} \mathcal F_Y$, giving the remaining equivalence.
\end{proof}
\begin{Cor}\label{cor:categorical-formal-structure-sheaf}
The endomorphism spectra of the tensor units define a sheaf
\[
\mathcal O^{\unit}_{\Spf(\cat K,Y)}
\coloneqq \End(\unit_{(-)})\circ\cat O^{\infty}_{\Spf(\cat K,Y)}
\colon\Open(Y)^{\op}\longrightarrow\CAlg(\Sp).
\]
There is a natural isomorphism of sheaves of Dirac rings
\[
a\bigl(\pi_*\mathcal O^{\unit}_{\Spf(\cat K,Y)}\bigr)
\cong\mathcal O_{\Spf(\cat K,Y)},
\]
where $\pi_*$ is applied sectionwise and $a$ denotes sheafification.
\end{Cor}
\begin{proof}
The functor $\End(\unit_{(-)})$ preserves limits by \Cref{remark:structure-sheaf-of-ring-spectra}, so the displayed composite is a sheaf. Let $\overline U=U\cap Y$ be a basic open, write $Z=\Spc(\cat K)\setminus U$, and let $L\colon\cat C\to\widehat{\cat C}_Y$ denote completion. Under the natural symmetric monoidal equivalence
\[
\cat O^{\infty}_{\Spf(\cat K,Y)}(\overline U)
\simeq\Mod_{Lf_Z}(\widehat{\cat C}_Y),
\]
the tensor unit corresponds $Lf_Z$. The free-forgetful adjunction and the completion adjunction therefore give natural equivalences
\begin{align*}
\mathcal O^{\unit}_{\Spf(\cat K,Y)}(\overline U)
&\simeq\Hom_{\widehat{\cat C}_Y}(L\unit,Lf_Z)\\
&\simeq\Hom_{\cat C}(\unit,\widehat{(f_Z)}_Y).
\end{align*}
These equivalences respect the commutative algebra structures and restriction maps: on the module-category side, restriction is extension of scalars along the canonical maps between the algebras $Lf_Z$. Since $f_Z\simeq\unit_U$, applying $\pi_*$ identifies the resulting presheaf on basic opens with ${}_{\mathrm p}\mathcal O_{\cat K,Y}$ from \Cref{def:formal-balmer}. Sheafification proves the claimed isomorphism.
\end{proof}
\section{The formal Hopkins--Neeman theorem}\label{sec:the_formal_hopkins_neeman_theorem}
Let $R$ be a noetherian ring and let $I \subseteq R$ be an ideal. In \Cref{exa:derived-complete-ring}, we introduced the derived category of $I$-complete $R$-modules, $\widehat{\Der(R)}_I$, as the completion of $\Der(R)$ along the closed subset $V(I) \subseteq \Spc(\Der(R)^c)$. In this section, we compute the formal spectrum $\Spf(\Der(R)^c,V(I))$ and identify it with the classical affine formal scheme $\Spf(R^{\wedge}_I)$.

\begin{Rem}\label{rem:completion-sequence}
We first recall some properties of the derived $I$-completion functor $\widehat{(-)}_I \colon \Der(R) \to \widehat{\Der(R)}_{I}$ from \Cref{exa:derived-complete-ring}. For a complex $M \in \Der(R)$, the object $\widehat{M}_I$ is the \emph{derived} $I$-adic completion of $M$. Even if $M$ is a discrete module, $\widehat{M}_I$ may differ from the classical $I$-adic completion. In general, \cite[Proposition~1.1]{GreenleesMay92} provides a short exact sequence for any discrete module $M$:
\[
    0 \to {\varprojlim_{k}}^{1} \Tor^R_{i+1}(R/I^k,M) \to \pi_i\widehat{M}_I \to \varprojlim_{k} \Tor^R_{i}(R/I^k,M)\to 0.
\]
This implies the following useful lemma:
\end{Rem}

\begin{Lem}\label{lem:completion-is-a-completion}
If $M$ is a flat discrete $R$-module, then the natural map $\widehat{M}_I \to M^{\wedge}_I$ is an isomorphism, where $M^{\wedge}_I$ denotes the ordinary $I$-adic completion.
\end{Lem}

Our version of the Hopkins--Neeman theorem is the following. The calculation in the proof shows that the Dirac structure sheaves appearing here are concentrated in degree zero; we therefore state the result in terms of ordinary locally ringed spaces.

\begin{Thm}\label{thm:formal-hn}
Let $R$ be a commutative noetherian ring and let $I \subseteq R$ be an ideal. Then there is an isomorphism of locally ringed spaces (and hence, formal schemes)
\[
\rho \colon \Spf\bigl(\Der(R)^c,V(I)\bigr) \xrightarrow{\sim} \Spf(R,I)
\]
induced by the comparison map of \Cref{thm:comparison-map}. Consequently, we have an isomorphism
\[
  \Spf\bigl(\widehat{\Der(R)}_I^{\dual}\bigr) \xrightarrow{\sim} \Spf(R^{\wedge}_I).
\]
\end{Thm}

\begin{proof}
By \Cref{rem:adic-completion} and \Cref{thm:completion-theorem}, it suffices to establish the isomorphism
\[
\rho \colon \Spf\bigl(\Der(R)^c,V(I)\bigr) \xrightarrow{\sim} \Spf(R,I).
\]
In this case, the stable homotopy theory associated to $\Der(R)^c$ is $\Der(R)$.

On the level of topological spaces, the result follows from \Cref{exa:balmer-scheme}: the isomorphism $\Spc(\Perf(R)) \cong \Spec(R)$ identifies the tensor-triangular support with the classical support of $R$-modules, mapping $V(I)$ homeomorphically to itself.

It remains to identify the structure sheaves. We compare the defining presheaves on the basis of open sets $U = \overline{D}(s) = D(s) \cap V(I)$ and then sheafify. Recall from \cite[Theorem~5.3]{Balmer10b} that under $\rho$, the preimage of $D(s)$ is the open set $U(\cone(s))$. Thus,
\[
\rho^{-1}\bigl(\overline{D}(s)\bigr) = U(\cone(s)) \cap V(I) = \overline{U}(\cone(s)).
\]
The corresponding value of the structure sheaf of $\Spf(R,I)$ is
\[
\mathcal O_{\Spf(R,I)}\bigl(\overline{D}(s)\bigr) \cong R[1/s]^{\wedge}_I.
\]
On the tensor-triangular side, since $R$ is the tensor unit of $\Der(R)$, the defining presheaf has value
\[
\begin{split}
{}_{\mathrm p}\mathcal O_{\Der(R)^c,V(I)}\bigl(\overline{U}(\cone(s))\bigr)
&= \pi_*\Hom_{\Der(R)}\bigl(R,\widehat{(\unit[\cone(s)^{-1}])}_I\bigr)\\
&\cong \pi_*\bigl(\widehat{(\unit[\cone(s)^{-1}])}_I\bigr).
\end{split}
\]
Here, $\unit[\cone(s)^{-1}]$ denotes the tensor unit in the localization of $\Der(R)$ at the open set $U(\cone(s))$. This localization is the Verdier quotient by the localizing tensor ideal $\Loc\langle\cone(s)\rangle$. Equivalently, extension of scalars along $R\to R[1/s]$ identifies it with $\Der(R[1/s])$; compare \cite[Theorem~3.3.7]{HoveyPalmieriStrickland97}. Under this equivalence, the localized unit $\unit[\cone(s)^{-1}]$ corresponds to the ring $R[1/s]$. Therefore, we have an equivalence of completions
\[
\widehat{(\unit[\cone(s)^{-1}])}_I \simeq \widehat{R[1/s]}_I.
\]
Since $R[1/s]$ is flat over $R$, \Cref{lem:completion-is-a-completion} gives an equivalence in $\Der(R)$
\[
\widehat{R[1/s]}_I\simeq R[1/s]^\wedge_I,
\]
where the right-hand side is regarded as a complex concentrated in degree zero. Consequently, there is a natural isomorphism of Dirac rings
\[
{}_{\mathrm p}\mathcal O_{\Der(R)^c,V(I)}\bigl(\overline U(\cone(s))\bigr)\cong R[1/s]^\wedge_I\cong\mathcal O_{\Spf(R,I)}\bigl(\overline D(s)\bigr).
\]
These isomorphisms are compatible with restriction and identify the morphism of presheaves induced by Balmer's comparison map with an objectwise isomorphism. After sheafification, they therefore give an isomorphism of structure sheaves. They also show that the Dirac structure sheaf of $\Spf(\Der(R)^c,V(I))$ is concentrated in degree zero. Thus $\rho$ is an isomorphism of ordinary locally ringed spaces. Since $\Spf(R,I)$ is an affine formal scheme, the source is an affine formal scheme as well, and $\rho$ is an isomorphism of formal schemes.
\end{proof}
\begin{Rem}
This result will also follow from results in \Cref{sec:a_comparison_map}, in particular \Cref{prop:comparison-map}. However we believe it is still instructive to do this special case.
\end{Rem}
\section{Globalizing to schemes}\label{sec:thomason}
We can now globalize the formal Hopkins--Neeman theorem (\Cref{thm:formal-hn}) to arbitrary noetherian schemes. We begin by recalling the notion of completion in the setting of schemes.

\begin{Def}
Let $X$ be a noetherian scheme, and let $Y \subseteq |X|$ be a closed subset. An \emph{ideal of definition} for $Y$ is a quasi-coherent ideal sheaf $\mathcal J \subseteq \mathcal{O}_X$ of finite type such that $|V(\mathcal J)| = Y$.
\end{Def}

\begin{Def}
Let $X$ be a quasi-compact quasi-separated noetherian scheme, and $Y \subseteq |X|$ a closed subset with ideal of definition $\mathcal J$.  For $n \ge 0$, set
\[
X_n \coloneqq (V(\mathcal{J}^{n+1}), \ \mathcal{O}_X / \mathcal{J}^{n+1}).
\]
The colimit in the category of locally ringed spaces
\[
\widehat{X}_Y \coloneqq \varinjlim_n X_n
\]
is called the \emph{formal completion} of $X$ along $Y$.
\end{Def}

\begin{Rem}
Explicitly, $\widehat{X}_Y$ is the locally ringed space whose underlying topological space is $Y$, equipped with the structure sheaf
\[
\mathcal{O}_{\widehat{X}_Y} \coloneqq \varprojlim_{n} \mathcal{O}_X / \mathcal{J}^{n+1}.
\]
By \cite[Proposition~10.8.5]{GrothendieckDieudonne1971Elements}, $\widehat{X}_Y$ is a formal scheme.
\end{Rem}
\begin{Exa}
Let $X = \Spec(R)$ and let $Y \subseteq |\Spec(R)|$ be a closed subset. Then there exists an ideal $I \subseteq R$ such that $V(I) = Y$, and one has natural isomorphisms
\[
\widehat{X}_Y \cong \Spf(R,I) \cong \Spf(R^{\wedge}_I).
\]
\end{Exa}

\begin{Def}\label{def:algebraizable}
A formal scheme is \emph{algebraizable} if it is of the form $\widehat{X}_Y$ for some noetherian scheme $X$ and a closed subset $Y \subseteq |X|$.
\end{Def}

\begin{Rem}
In the rest of this section, we will restrict our attention to algebraizable formal schemes. Even in the noetherian case, there exist non-algebraizable formal schemes, although they are difficult to construct; see \cite[Section 5]{HironakaMatsumura1968Formal}.
\end{Rem}
For the following result, the affine calculation of \Cref{thm:formal-hn} shows that the relevant Dirac structure sheaves are concentrated in degree zero. We may therefore regard the Dirac ringed spaces appearing below as ordinary ringed spaces.
\begin{Thm}\label{thm:global-formal-hn}
Let $X$ be a noetherian scheme, and let $Z \subseteq X$ be a closed subset. Let $\widehat{X}_Z$ denote the formal completion of $X$ along $Z$. Then there is an isomorphism of locally ringed spaces (and hence formal schemes)
\[
\Phi \colon \Spf\bigl(\Der_{\qc}(X)^c, Z\bigr) \xrightarrow{\sim} \widehat{X}_Z.
\]
\end{Thm}
\begin{proof}
Choose an ideal of definition $\mathcal J\subseteq\mathcal O_X$ for $Z$. Let $\{U_{\alpha}\}$ be a finite affine open cover of $X$, where $U_{\alpha}=\Spec(R_{\alpha})$. Since $U_{\alpha}$ is affine, it is a quasi-compact open subset of $X$. Let $Z_{\alpha}\coloneqq Z\cap U_{\alpha}$ and let $I_{\alpha}\coloneqq\Gamma(U_\alpha,\mathcal J|_{U_\alpha})\subseteq R_\alpha$, so that $Z_{\alpha}=V(I_{\alpha})$.

By the definition of the formal completion of a scheme, $\widehat{X}_Z$ is obtained by gluing the affine formal schemes
\[
(\widehat{X}_Z)|_{Z_{\alpha}}\cong\Spf(R_{\alpha},I_{\alpha}).
\]
On the tensor-triangular side, we identify $\Spc(\Der_{\qc}(X)^c)\cong X$ as topological spaces (\Cref{exa:balmer-scheme}). Since $U_{\alpha}$ is a quasi-compact open set, we may apply the locality theorem (\Cref{Thm:spf-open-restriction}) to deduce
\[
\Spf\bigl(\Der_{\qc}(X)^c,Z\bigr)\big|_{Z_{\alpha}}\cong\Spf\bigl(\Der_{\qc}(X)^c(U_{\alpha}),Z_{\alpha}\bigr).
\]
Recall from \Cref{rem:verdier-quotient} that $\Der_{\qc}(X)^c(U_{\alpha})$ denotes the Karoubi quotient. By the Thomason--Trobaugh localization theorem \cite[Theorem~5.2.2]{ThomasonTrobaugh90} (see also \cite[Theorem~2.13]{Balmer02}), there is an equivalence
\[
\Der_{\qc}(X)^c(U_{\alpha})\simeq\Der_{\qc}(U_{\alpha})^c\simeq\Der(R_{\alpha})^c.
\]
Thus, we have an isomorphism of Dirac ringed spaces
\[
\Spf\bigl(\Der_{\qc}(X)^c,Z\bigr)\big|_{Z_{\alpha}}\cong\Spf\bigl(\Der(R_{\alpha})^c,V(I_{\alpha})\bigr).
\]
Now we apply the affine formal Hopkins--Neeman theorem (\Cref{thm:formal-hn}). This provides an isomorphism of ringed spaces
\[
\rho_{\alpha}\colon\Spf\bigl(\Der(R_{\alpha})^c,V(I_{\alpha})\bigr)\xrightarrow{\sim}\Spf(R_{\alpha},I_{\alpha}).
\]
Combining these identifications, we obtain local isomorphisms
\[
\phi_{\alpha}\colon\Spf\bigl(\Der_{\qc}(X)^c,Z\bigr)\big|_{Z_{\alpha}}\xrightarrow{\sim}(\widehat{X}_Z)|_{Z_{\alpha}}.
\]
We verify that these local isomorphisms agree on overlaps by restricting to an affine basis. Let $W\coloneqq\Spec(R_W)\subseteq X$ be any affine open and set $I_W\coloneqq\Gamma(W,\mathcal J|_W)$, so that $Z\cap W=V(I_W)$. The locality theorem, Thomason--Trobaugh localization, and the affine formal Hopkins--Neeman theorem give a canonical composite
\[
\begin{split}
\rho_W\colon
\Spf\bigl(\Der_{\qc}(X)^c,Z\bigr)\big|_{Z\cap W}
&\xrightarrow{\sim}
\Spf\bigl(\Der(R_W)^c,V(I_W)\bigr)\\
&\xrightarrow{\sim}
\Spf(R_W,I_W)
\xrightarrow{\sim}
(\widehat X_Z)|_{Z\cap W}.
\end{split}
\]
This map depends only on $W$. Indeed, Balmer's comparison map is natural with respect to localization \cite[Theorem~5.3]{Balmer10b}, the open-restriction identifications are transitive by \Cref{lem:open-restriction-transitivity}, the Thomason--Trobaugh equivalences are induced by restriction, and the ideals $I_W$ are obtained by restricting the fixed ideal sheaf $\mathcal J$. Consequently, if $W\subseteq U_\alpha$ is affine, then
\[
\phi_\alpha|_{Z\cap W}=\rho_W.
\]
Now let $U_{\alpha\beta}=U_\alpha\cap U_\beta$. For every affine open $W\subseteq U_{\alpha\beta}$, both $\phi_\alpha$ and $\phi_\beta$ restrict on $Z\cap W$ to $\rho_W$. Since affine opens form a basis and a morphism of sheaves is determined on a basis, the restrictions of $\phi_\alpha$ and $\phi_\beta$ to $Z\cap U_{\alpha\beta}$ agree. As the $\phi_\alpha$ are isomorphisms onto the restrictions of the fixed Dirac ringed space $\widehat X_Z$ and agree on overlaps, they glue to an isomorphism of Dirac ringed spaces
\[
\Phi\colon\Spf\bigl(\Der_{\qc}(X)^c,Z\bigr)\xrightarrow{\sim}\widehat X_Z.
\]
The affine formal Hopkins--Neeman theorem also shows that the restriction of the structure sheaf of $\Spf(\Der_{\qc}(X)^c,Z)$ to every $Z_\alpha$ is concentrated in degree zero. Since the $Z_\alpha$ cover $Z$, the entire structure sheaf is concentrated in degree zero. Thus $\Phi$ is an isomorphism of ordinary locally ringed spaces. Since $\widehat X_Z$ is a formal scheme, the source is a formal scheme as well, and $\Phi$ is an isomorphism of formal schemes.
\end{proof}

\begin{Rem}
One can ask for extensions beyond schemes (e.g., to certain algebraic stacks or other locally affine tensor-triangular situations). Conceptually, the argument above only uses (i) an affine identification (a formal Hopkins--Neeman statement), (ii) a locality statement for $\Spf$ under restriction to quasi-compact opens, via \Cref{Thm:spf-open-restriction} and (iii) a Thomason--Trobaugh type localization theorem identifying the relevant Verdier quotients with the corresponding ``affine'' categories on opens. Any setting in which analogues of (i)--(iii) hold admits the same gluing argument.
\end{Rem}
\begin{Def}\label{def:Dqc-complete}
Let $\widehat{X}_Z$ denote the formal completion of $X$ along $Z$ as above. We set
\[
  \Der^{\wedge}_{\qc}(\widehat{X}_Z)\coloneqq \widehat{\Der_{\qc}(X)}_Z,
\]
the completion of $\Der_{\qc}(X)$ at the Thomason subset $Z\subseteq \Spc(\Der_{\qc}(X)^c)\cong X$ in the sense of \Cref{def:completion}.
\end{Def}

\begin{Cor}\label{cor:global-formal-hn}
Let $X$ be a noetherian scheme and let $Z\subseteq X$ be a closed subset. Let $\widehat{X}_Z$ denote the formal completion of $X$ along $Z$. Then there is an isomorphism of ringed spaces (and hence formal schemes)
\[
  \Phi\colon \Spf\bigl(\Der^{\wedge}_{\qc}(\widehat{X}_Z)^{\dual}\bigr)\xrightarrow{\sim}\widehat{X}_Z.
\]
\end{Cor}
\begin{proof}
By \Cref{thm:completion-theorem}, the source is isomorphic to $\Spf(\Der_{\qc}(X)^c,Z)$, which is isomorphic to $\widehat X_Z$ by \Cref{thm:global-formal-hn}.
\end{proof}
\begin{Rem}\label{rem:completion-as-local-homology}
Recall from \Cref{rem:completion-formula} that completion at a Thomason subset $Z\subseteq \Spc(\Der_{\qc}(X)^c)\cong X$ is given by the endofunctor
\[
  \widehat{(-)}_Z\simeq \ihom{e_Z,-},
\]
where $e_Z$ is the Balmer--Favi idempotent associated to $Z$.

When $X$ is separated, one can identify this idempotent with a familiar geometric object: there is an equivalence
\[
  e_Z\simeq \mathbb{R}\Gamma_Z(\mathcal{O}_X)
\]
in $\Der_{\qc}(X)$, where $\mathbb{R}\Gamma_Z$ denotes the right derived functor of $\Gamma_Z$, the subsheaf with supports in $Z$, see \cite[Lemma~5.4]{Stevenson2018Filtrations}. In particular, in this case the abstract completion functor $\ihom{e_Z,-}$ agrees with the derived completion (local homology) functor of Alonso Tarr\'io--Jerem\'ias L\'opez--Lipman: by \cite[Theorem~(0.3)]{AlonsoJeremiasLipman97}, the right adjoint to $\mathbb{R}\Gamma_Z$ is naturally isomorphic to $\mathbb{L}\Lambda_Z$, and under the identification $e_Z\simeq \mathbb{R}\Gamma_Z(\mathcal{O}_X)$ this yields a natural equivalence of endofunctors on $\Der_{\qc}(X)$
\[
  \widehat{(-)}_Z\simeq \ihom{\mathbb{R}\Gamma_Z(\mathcal{O}_X),-}\simeq \mathbb{L}\Lambda_Z(-).
\]
\end{Rem}

\begin{Rem}\label{rem:AJL-comparison-sketch}
We briefly indicate how \Cref{def:Dqc-complete} compares with the triangulated category denoted $\Der^{\wedge}_{\qc}(\widehat{X}_Z)$ in \cite[\S6.3]{alonso1999duality}, at least when $X$ is separated. This comparison is not used elsewhere in the paper.

Write $\widetilde{\Der}^{\wedge}_{\qc}(\widehat{X}_Z)$ for the category constructed in \cite[Remark 6.3.1]{alonso1999duality}. Alonso--Jerem\'ias--Lipman show that on the formal scheme $\widehat{X}_Z$ the ``complete'' and ``torsion'' variants agree, and identify the torsion category on $\widehat{X}_Z$ with the $Z$-supported subcategory on $X$:
\[
  \widetilde{\Der}^{\wedge}_{\qc}(\widehat{X}_Z)
  \simeq \Der_{\qct}(\widehat{X}_Z)
  \simeq \Der_{\qc Z}(X),
\]
where $\Der_{\qc Z}(X)$ denotes the essential image of $\mathbb{R}\Gamma_Z$ on $\Der_{\qc}(X)$. (For precise statements and hypotheses, see \cite[\S6.3]{alonso1999duality} and \cite[Proposition~5.2.4]{alonso1999duality}.)

On the other hand, \Cref{rem:completion-as-local-homology} identifies our completion functor $\widehat{(-)}_Z$ with $\mathbb{L}\Lambda_Z$ (when $X$ is separated). A form of Greenlees--May duality identifies the torsion and complete pieces as equivalent subcategories of $\Der_{\qc}(X)$, so the essential image of $\mathbb{L}\Lambda_Z$ agrees with the completion subcategory $\widehat{\Der_{\qc}(X)}_Z$ of \Cref{def:Dqc-complete}. Under these identifications, one obtains an equivalence
\[
  \widetilde{\Der}^{\wedge}_{\qc}(\widehat{X}_Z)\simeq \widehat{\Der_{\qc}(X)}_Z.
\]
In fact, the formulas in \cite[Proposition 5.2.4]{alonso1999duality} imply that this equivalence is given by $\mathbb{L}\Lambda_Z \circ k_*$ where $k \colon \widehat{X}_Z \to X$ is the completion map.
\end{Rem}
\section{A comparison map}\label{sec:a_comparison_map}

One of the most important tools in computing the Balmer spectrum of a tensor-triangulated category is Balmer's comparison map (\Cref{thm:comparison-map}). A general strategy for computing $\Spc(\cat K)$ is to first determine $\Spec^h(\pi_*\Hom_{\cat K}(\unit,\unit))$ and then compute the fibers of the comparison map; see \cite{BalmerSanders17} and \cite{abhs} for examples. It is therefore natural to seek an analogue of the comparison map for the formal spectrum introduced in this paper. Throughout this section we assume $R_* \coloneqq \pi_*\Hom_{\cat K}(\unit,\unit)$ is noetherian, which implies that $\rho$ is surjective \cite[Theorem~7.3]{Balmer10b}. In this section we will construct a comparison map for `algebraic' completions, i.e., those coming from an ideal in $R_*$.
\begin{Def}
Let $\cat K$ be a 2-ring and let $I \subseteq R_*$ be a homogeneous ideal. Write $Y(I) \coloneqq \rho^{-1}(V(I))$ for the preimage of the closed subset $V(I) \subseteq \Spec^h(R_*)$ under the comparison map of \Cref{thm:comparison-map}. Define
\[
  \widehat{ \cat C}_I \coloneqq \widehat{\cat C}_{Y(I)}
\]
to be the completion of $\cat C$ at the closed subset $Y(I) \subseteq \Spc(\cat K)$ in the sense of \Cref{def:completion}.
\end{Def}

\begin{Exa}
If $\cat K = \Perf(R)$ for a commutative noetherian ring $R$, then $\rho$ is a homeomorphism, and we are exactly in the situation of \Cref{exa:derived-complete-ring}.
\end{Exa}

\begin{Rem}
The spectrum $R \coloneqq \Hom_{\cat C}(\unit,\unit)$ is a commutative algebra in spectra and, for any $X\in\cat C$, $\Hom_{\cat C}(\unit,X)$ is an $R$-module. By Morita theory there is a symmetric monoidal adjunction
% https://q.uiver.app/#q=WzAsMixbMCwwLCJGIFxcY29sb24gXFxNb2RfUiJdLFsxLDAsIlxcY2F0IEMgXFxub2xvYyBHIl0sWzAsMV0sWzIsMCwiIiwwLHsib2Zmc2V0IjoxLCJsZXZlbCI6MX1dXQ==
\begin{equation}\begin{tikzcd}[cramped]\label{eq:morita}
    {F \colon \Mod_R} & {\cat C \noloc G}
    \arrow[shift left, from=1-1, to=1-2]
    \arrow[shift left, from=1-2, to=1-1]
\end{tikzcd}\end{equation}
with $F(M)=M\otimes_R\unit$ and $G(X)=\Hom_{\cat C}(\unit,X)$. Since $\cat C=\Ind(\cat K)$ is rigidly-compactly generated and $\unit$ is compact, $F$ is fully faithful and $G$ exhibits $\Mod_R$ as a colocalization of $\cat C$.

Let $\rho_{\cat K}\colon\Spc(\cat K)\to\Spec^h(R_*)$ and $\rho_R\colon\Spc(\Mod_R^c)\to\Spec^h(R_*)$ be the respective comparison maps. Functoriality of the comparison map \cite[Theorem 5.3]{Balmer10b} yields a commutative diagram
\[
\begin{tikzcd}[cramped]
    {\Spc(\cat K)} \arrow[d, two heads, "\rho_{\cat K}"'] \arrow[r, "{\Spc(F)}"]
    & {\Spc(\Mod_R^c)} \arrow[d, two heads, "\rho_R"] \\
    {\Spec^h(R_*)} \arrow[r, equals] & {\Spec^h(R_*)}
\end{tikzcd}
\]
and hence, for any homogeneous ideal $I\subseteq R_*$, the associated Thomason subsets
\[
Y(I) \coloneqq \rho_{\cat K}^{-1}\bigl(V(I)\bigr)\subseteq\Spc(\cat K),
\qquad
Y_R(I)\coloneqq \rho_R^{-1}\bigl(V(I)\bigr)\subseteq\Spc(\Mod_R^c)
\]
satisfy
\[
Y(I)=\Spc(F)^{-1}\!\bigl(Y_R(I)\bigr).
\]
Let $e_{Y(I)}\in\cat C$ and $e_{Y_R(I)}\in\Mod_R$ be the associated tensor idempotents. By functoriality \cite[Proposition 5.11]{BalmerSanders17} we have $F(e_{Y_R(I)})\cong e_{Y(I)}$, and since $F$ is fully faithful, $G(e_{Y(I)})\cong e_{Y_R(I)}$ in $\Mod_R$.
\end{Rem}

\begin{Prop}\label{prop:morita-eYI}
For every $X\in\cat C$ there is a natural isomorphism of $R$-modules
\[
\Hom_{\cat C}(e_{Y(I)},X) \cong
\Hom_{\Mod_R}\!\bigl(e_{Y_R(I)},\Hom_{\cat C}(\unit,X)\bigr).
\]
\end{Prop}
\begin{proof}
Using $F(e_{Y_R(I)})\cong e_{Y(I)}$ and the adjunction \eqref{eq:morita},
\[
\Hom_{\cat C}(e_{Y(I)},X)
\cong \Hom_{\cat C}\bigl(F(e_{Y_R(I)}),X\bigr)
\cong \Hom_{\Mod_R}\bigl(e_{Y_R(I)},\Hom_{\cat C}(\unit,X)\bigr)
\]
as required.
\end{proof}

\begin{Prop}\label{prop:comparison-map}
Let $\cat K\in\tring$ with $R_*$ noetherian, and let $I\subseteq R_*$ be a homogeneous ideal. Set $Y(I)\coloneqq\rho^{-1}(V(I))\subseteq\Spc(\cat K)$, where $\rho\colon\Spc(\cat K)\to\Spec^h(R_*)$ is Balmer's comparison map of \Cref{thm:comparison-map}. Then there is a natural morphism of Dirac ringed spaces
\[
  \overline{\rho}\colon \Spf(\cat K,Y(I))\longrightarrow \Spf^h(R_*,I),
\]
whose underlying map of spaces is the restriction $\rho|_{Y(I)}\colon Y(I)\to V(I)$. If the underlying comparison map $\rho$ is a homeomorphism, then $\overline{\rho}$ is an isomorphism of Dirac ringed spaces.
\end{Prop}
\begin{proof}
On underlying spaces, we take the restriction $\overline{\rho}\coloneqq\rho|_{Y(I)}\colon Y(I)\to V(I)$. Let $\overline D(s)=D(s)\cap V(I)$, where $s\in R_*$ is homogeneous. By \cite[Theorem~5.3(b)]{Balmer10b}, $\rho^{-1}(D(s))=U(\cone(s))$, and hence
\[
\overline{\rho}^{-1}\bigl(\overline D(s)\bigr)=Y(I)\cap\rho^{-1}(D(s))=\overline U(\cone(s)).
\]
The defining presheaf of $\Spf^h(R_*,I)$ takes the value
\[
\bigl(R_*[1/s]\bigr)^{\wedge}_I
\]
on $\overline D(s)$. On the tensor-triangular side, \Cref{rem:adjoint-sheaf,prop:morita-eYI} identify the corresponding value of the defining presheaf with
\[
{}_{\mathrm p}\mathcal O_{\cat K,Y(I)}\bigl(\overline U(\cone(s))\bigr)\cong\pi_*\Hom_{\Mod_R}\bigl(e_{Y_R(I)},\Hom_{\cat C}(\unit,\unit[1/s])\bigr).
\]
We first identify the completion appearing here with the derived completion in the sense of Greenlees and May. Choose homogeneous generators $x_1,\ldots,x_n$ of $I$ and set
\[
K(I)\coloneqq\cone(x_1)\otimes_R\cdots\otimes_R\cone(x_n)\in\Perf(R).
\]
Using $\rho_R^{-1}(V(x_i))=\supp_R(\cone(x_i))$, we obtain
\[
Y_R(I)=\rho_R^{-1}(V(I))=\bigcap_{i=1}^n\supp_R(\cone(x_i))=\supp_R(K(I)).
\]
It follows that the $Y_R(I)$-torsion subcategory of $\Mod_R$ is the localizing tensor ideal generated by $K(I)$. The Balmer--Favi torsion idempotent $e_{Y_R(I)}$ therefore identifies with the stable Koszul, or local-cohomology, object $\Gamma_I R$. Abstract local duality gives a natural equivalence
\[
\Hom_{\Mod_R}(e_{Y_R(I)},M) \simeq \Hom_{\Mod_R}(\Gamma_I R,M) \simeq\Lambda^I M,
\]
where $\Lambda^I$ denotes Greenlees--May derived $I$-adic completion; see \cite{bhv1,bhv2}. Consequently,
\[
\pi_*\Hom_{\Mod_R}\bigl(e_{Y_R(I)},M\bigr)\cong\pi_*\Lambda^I M.
\]
Set $M_s\coloneqq\Hom_{\cat C}(\unit,\unit[1/s])$. Since the unit is compact and $\unit[1/s]$ is obtained by telescope localization, the canonical map
\[
R_*[1/s]\longrightarrow\pi_*M_s
\]
is an isomorphism of Dirac rings. Under the noetherian hypothesis, the Greenlees--May spectral sequence is strongly convergent and has $E_2$-page
\[
E_2^{p,q}=H^I_{-p,-q}\bigl(\pi_*M_s\bigr)\cong H^I_{-p,-q}\bigl(R_*[1/s]\bigr),
\]
converging to $\pi_*\Lambda^I M_s$; see \cite[Eq.~(3.3) and Theorem~4.2]{GreenleesMay1995Completions}. Here $H^I_{*,*}(-)$ denotes graded local homology, which agrees with the left derived functors of graded $I$-adic completion by \cite[Theorem~6.3]{Miller2002Graded}. Since $R_*[1/s]$ is graded-flat over $R_*$, the graded form of \cite[Proposition~1.1]{GreenleesMay92}, applied degreewise, gives
\[
H^I_{i,*}\bigl(R_*[1/s]\bigr)=0\quad(i>0),\qquad H^I_{0,*}\bigl(R_*[1/s]\bigr)\cong\varprojlim_nR_*[1/s]/I^nR_*[1/s].
\]
The spectral sequence therefore collapses and its edge morphism gives a natural isomorphism
\[
\pi_*\Lambda^I M_s\xrightarrow{\sim}\bigl(R_*[1/s]\bigr)^{\wedge}_I.
\]
This is an isomorphism of Dirac rings. Indeed, $M_s$ is a commutative $R$-algebra, derived completion is lax symmetric monoidal, and the Greenlees--May spectral sequence is multiplicative. Taking the inverse of the edge isomorphism yields
\[
\bigl(R_*[1/s]\bigr)^{\wedge}_I\xrightarrow{\sim}{}_{\mathrm p}\mathcal O_{\cat K,Y(I)}\bigl(\overline U(\cone(s))\bigr).
\]
These maps are natural in the homogeneous element $s$ and compatible with restriction. After sheafification, they define
\[
\overline{\rho}^{\#}\colon\mathcal O_{\Spf^h(R_*,I)}\longrightarrow\overline{\rho}_*\mathcal O_{\Spf(\cat K,Y(I))},
\]
and hence a morphism of Dirac ringed spaces
\[
\overline{\rho}\colon\Spf(\cat K,Y(I))\longrightarrow\Spf^h(R_*,I).
\]
If $\rho$ is a homeomorphism, then $\overline{\rho}$ is a homeomorphism and the preceding maps are isomorphisms on a basis. Their sheafifications are therefore isomorphisms, so $\overline{\rho}$ is an isomorphism of Dirac ringed spaces.
\end{proof}
\begin{Cor}\label{cor:complete-comparison}
Under the hypotheses above, there is a natural morphism of Dirac ringed spaces
\[
  \widehat{\rho} \colon \Spf\bigl(\widehat{\cat C}_{Y(I)}^{\dual}\bigr)
  \longrightarrow \Spf^h\bigl((R_*)^{\wedge}_I\bigr).
\]
\end{Cor}
\begin{proof}
Combine \Cref{thm:completion-theorem} with \Cref{rem:adic-completion} and \Cref{prop:comparison-map}.
\end{proof}

\begin{Exa}
Let $A$ be an even periodic commutative ring spectrum with $\pi_0(A)$ a complete regular local ring with maximal ideal $\mathfrak m$ . By \cite[Theorem~1.1]{DellAmbrogioStanley16} the comparison map
\[
  \rho\colon \Spc(\Perf(A)) \longrightarrow \Spec^h(\pi_*A) \cong \Spec(\pi_0A)
\]
is a homeomorphism. In fact, by \cite[Proposition 6.11]{Balmer10b}, this is even an isomorphism of Dirac ringed spaces. Let $\kappa(A)\coloneqq A/\mathfrak m$ (so that $\pi_0\kappa(A)\cong \pi_0(A)/\mathfrak m$). This is an $A$-module with $\supp(\kappa(A))=V(\mathfrak m)$. Applying \Cref{prop:comparison-map} and \Cref{cor:complete-comparison} with $I=\mathfrak m$ we obtain an isomorphism of Dirac ringed spaces
\[
  \Spf\bigl((L_{\kappa(A)}\Mod_A)^{\dual}\bigr)
  \xrightarrow{ \cong } \Spf^h(\pi_*(A),\mathfrak m).
\]
Here we use \Cref{rem:completion-as-bousfield} to identify the completion functor as a Bousfield localization in this case.
\end{Exa}

\section{Chromatic homotopy}\label{sec:chromatic}
We now turn to examples coming from chromatic homotopy theory. As usual, we fix a prime number $p$ throughout.

\begin{Rem}\label{rem:BP-En}
Recall that the Brown--Peterson spectrum $BP$ has coefficient ring
\[
    BP_* \cong \mathbb{Z}_{(p)}[v_1,v_2,\ldots].
\]
Johnson--Wilson theory $E(n)$ is obtained from $BP$ by first quotienting by the ideal $(v_{n+1},v_{n+2},\ldots)$ and then inverting $v_n$ (using, for example, highly structured ring spectra \cite{ElmendorfKrizMellMay1997Rings}). Its coefficients are
\[
    E(n)_* \cong v_n^{-1}BP_*/(v_{n+1},v_{n+2},\ldots).
\]
More generally, for an invariant regular sequence $J_k=(p^{i_0},v_1^{i_1},\ldots,v_{k-1}^{i_{k-1}})$, we can form a quotient $BP/J_k$ with $\pi_*(BP/J_k) \cong BP_*/J_k$, and for $0\le k\le n$ define a spectrum $E(n,J_k)$ with
\[
    E(n,J_k)_* \cong v_n^{-1}\!\bigl((BP/J_k)_*/(v_{n+1},v_{n+2},\ldots)\bigr).
\]
For $k=0$ this recovers $E(n)$, while for $J_n=(p,v_1,\ldots,v_{n-1})$ we have $E(n,J_n)\simeq K(n)$, Morava $K$-theory. Moreover, by \cite[Proposition~2.10]{Heard2023local}, there is an equivalence of Bousfield classes
\[
    \langle E(n,J_k)\rangle =\langle K(k)\vee\cdots\vee K(n)\rangle,
\]
hence
\[
    \Sp_{E(n,J_k)} \simeq \Sp_{K(k)\vee\cdots\vee K(n)}.
\]
\end{Rem}

\begin{Not}\label{not:Spn-notation}
Write $\Sp_n\coloneqq \Sp_{E(n)}$ for the $E(n)$-local stable homotopy category, $\Sp_n^c$ for its subcategory of compact objects, and $L_n$ for Bousfield localization at $E(n)$. For $h\ge 0$ let $K(h)$ denote Morava $K$-theory and fix a finite type $(h+1)$-spectrum $F(h+1)$.
\end{Not}

\begin{Def}\label{def:Ph}
For $0\le h\le n$, define the $\otimes$-ideal
\[
    \cat P_h \coloneqq \SET{x\in \Sp_n^c}{K(h)_*(x)=0}.
\]
\end{Def}

\begin{Rem}\label{rem:Ph-prime}
Since $K(h)_*(x\otimes y)\cong K(h)_*(x)\otimes_{K(h)_*}K(h)_*(y)$, each $\cat P_h$ is prime. Moreover,
\[
    \cat P_h = \thickid{\langle L_nF(h+1)\rangle}.
\]
\end{Rem}

\begin{Thm}[Hovey--Strickland]\label{thm:spec-en}
The spectrum of the $E(n)$-local category is the finite chain
\[
    \Spc(\Sp_n^c) = \{ \underbrace{\cat P_0}_{\text{generic}} \supsetneq \cat P_1 \supsetneq \cdots \supsetneq \underbrace{\cat P_n}_{\text{closed}} \}.
\]
The topology is defined by the closure operator $\overline{\{\cat P_h\}} = \SET{\cat P_k}{h\le k\le n}$. For $0\le k\le n$, consider the basic open set defined by the type $(k+1)$-spectrum $F(k+1)$:
\[
    U_k \coloneqq U\bigl(L_nF(k+1)\bigr)
    = \SET{\cat P_i}{L_nF(k+1) \in \cat P_i}
    = \SET{\cat P_i}{0 \le i \le k}.
\]
The value of the structure sheaf on this open set is
\[
    \mathcal O_{\Spc(\Sp_n^c)}(U_k) \;\cong\; \pi_* L_k S^0.
\]
\end{Thm}

\begin{proof}
The topological description is due to \cite[Theorem~6.9]{HoveyStrickland99}; see also \cite[Proposition~3.5]{BarthelHeardNaumann20pp} for a translation into this language.

The point $\cat P_k$ has no proper open neighborhood in $U_k$, so every open cover of $U_k$ contains $U_k$ itself and sheafification does not change the defining presheaf value. By \cite[Corollary~6.10]{HoveyStrickland99}, the finite localization of $\Sp_n$ away from $L_nF(k+1)$ identifies the localized category $\Sp_n[L_nF(k+1)^{-1}]$ with $\Sp_k$. Under this identification the localized tensor unit corresponds to $L_kS^0$. Therefore,
\[
    \mathcal O_{\Spc(\Sp_n^c)}(U_k)
    \cong \pi_* \End_{\Sp_k}(L_kS^0)
    \cong \pi_* L_k S^0.\qedhere
\]
\end{proof}

\begin{Def}\label{def:Yhplus1}
For a fixed height $h<n$, let $Y_{h+1}$ be the closed subset corresponding to spectra of type $\ge h+1$:
\[
    Y_{h+1} \coloneqq \supp(L_nF(h+1)) = \{ \cat P_{h+1}, \ldots, \cat P_n \}.
\]
\end{Def}

\begin{Rem}\label{rem:spk-n}
By \Cref{rem:completion-as-bousfield}, the completion $(\widehat{\Sp_n})_{Y_{h+1}}$ is equivalent to Bousfield localization of $\Sp_n$ at $L_nF(h+1)$. By \cite[Lemma~6.16]{bhv1}, this is equivalent to localization of $\Sp$ at $F(h+1)\otimes E(n)$, which by \cite[Proposition~2.10]{Heard2023local} is equivalent to $\Sp_{K(h+1)\vee\cdots\vee K(n)}$ (the special case where $h = n-1$ is considered in \cite[Proposition 8.11]{balmer2025tateintermediatevaluetheorem}). Moreover, \cite[Corollary~2.29]{Heard2023local} gives
\[
    L_nF(h+1)\simeq L_{K(h+1)\vee\cdots\vee K(n)} F(h+1).
\]
By \Cref{thm:completion-theorem}, we therefore obtain an isomorphism of Dirac ringed spaces
\[
    \Spf(\Sp_n^{\dual},Y_{h+1}) \cong 
    \Spf\bigl(\Sp_{K(h+1)\vee\cdots\vee K(n)}^{\dual}\bigr).
\]
\end{Rem}

\begin{Thm}\label{thm:formal-chromatic}
For $0\le h\le n-1$, the formal spectrum \( \Spf\bigl(\Sp_{K(h+1)\vee\cdots\vee K(n)}^{\dual}\bigr) \) has underlying topological space
\[
    Y_{h+1} = \{ \cat P_{h+1} \supsetneq \cat P_{h+2} \supsetneq \cdots \supsetneq \cat P_n \} \subseteq \Spc(\Sp_n^c).
\]
For $h+1 \le k \le n$, on the basic open set
\[
    \overline{U}_k \coloneqq U(L_nF(k+1)) \cap Y_{h+1} = \SET{\cat P_i}{h+1 \le i \le k},
\]
the structure sheaf takes the value\footnote{We omit the subscript on the $\mathcal O$ for reasons of space.}
\[
    \mathcal O\bigl(\overline{U}_k\bigr)
   \cong \pi_*L_{K(h+1)\vee\cdots\vee K(k)} S^0.
\]
\end{Thm}

\begin{proof}
The underlying space statement follows from \Cref{thm:spec-en} and the definition of $Y_{h+1}$.

The point $\cat P_k$ has no proper open neighborhood inside $\overline U_k$. Indeed, any open subset of $Y_{h+1}$ containing $\cat P_k$ contains all of $\overline U_k$. Hence every open cover of $\overline U_k$ contains $\overline U_k$ itself, so sheafification does not change the defining presheaf value on this basic open.

For $h+1\le k\le n$, the Bousfield class argument in \Cref{rem:spk-n} shows that
\[
\begin{split}
\mathcal O\bigl(\overline{U}_k\bigr)
&\cong \pi_* \Hom_{\Sp_n}\bigl(L_n S^0,
L_{K(h+1)\vee\cdots\vee K(n)}\,L_k S^0\bigr) \\
&\cong \pi_* L_{K(h+1)\vee\cdots\vee K(n)}L_k S^0.
\end{split}
\] 

It remains to show that
\[
L_{K(h+1)\vee\cdots\vee K(n)}L_k S^0 \simeq L_{K(h+1)\vee\cdots\vee K(k)} S^0.
\]
For $k=n$ this follows directly from the final localization argument below, so assume that $k<n$. By \cite[Remark~2.27]{Heard2023local}, there is a pullback square \[\begin{tikzcd}[ampersand replacement=\&]
    {L_{K(h+1) \vee \cdots \vee K(n)}L_kS^0} \& {L_{K(h+1) \vee \cdots \vee K(k)}L_kS^0} \\
    {L_{K(k+1) \vee \cdots \vee K(n)}L_kS^0} \& {L_{K(h+1) \vee \cdots \vee K(k)}L_{K(k+1) \vee \cdots \vee K(n)}L_kS^0.} \arrow[from=1-1, to=2-1] \arrow[from=1-1, to=1-2] \arrow[from=2-1, to=2-2] \arrow[from=1-2, to=2-2]
\end{tikzcd}\]
Since $L_k$ is smashing, $\langle L_kS^0\rangle=\langle E(k)\rangle=\langle K(0)\vee\cdots\vee K(k)\rangle$. Thus $K(i)\wedge L_kS^0\simeq0$ for $i>k$, and hence $L_{K(k+1)\vee \cdots \vee K(n)}L_kS^0\simeq 0$. The bottom-right corner also vanishes, being a further localization of the bottom-left corner. Therefore, the pullback identifies the upper-left corner with the upper-right corner:
\[
L_{K(h+1) \vee \cdots \vee K(n)}L_kS^0 \;\simeq\; L_{K(h+1) \vee \cdots \vee K(k)}L_kS^0.
\]
Finally, since $\langle K(h+1)\vee\cdots\vee K(k)\rangle\le \langle E(k)\rangle$, localization at $E(k)$ does not affect $K(h+1)\vee\cdots\vee K(k)$-localization, and we obtain
\[
L_{K(h+1) \vee\cdots \vee K(k)} L_kS^0 \simeq L_{K(h+1)\vee \cdots \vee K(k)}S^0.
\]
Combining these equivalences yields the claimed identification, and hence
\[
\mathcal O\bigl(\overline{U}_k\bigr)
\cong \pi_*L_{K(h+1)\vee\cdots\vee K(k)} S^0,
\]
as desired.
\end{proof}
\begin{Rem}
The categorical structure sheaf of \Cref{thm:categorical-formal-sheaf} refines the preceding calculation. For $h+1 \le k \le n$, there are natural symmetric monoidal equivalences
\[
\cat O_{\Spf(\Sp_n^c,Y_{h+1})}^{\infty}\bigl(\overline U_k\bigr)
\simeq \widehat{\Sp_k}_{Y_{h+1} \cap U_k}
\simeq \Sp_{K(h+1) \vee \cdots \vee K(k)}.
\]
Indeed, $\Sp_n(U_k) \simeq \Sp_k$, and completion at $Y_{h+1} \cap U_k$ is Bousfield localization at $K(h+1) \vee \cdots \vee K(k)$. Consequently, the associated sheaf of commutative algebra spectra satisfies
\[
\mathcal O_{\Spf(\Sp_n^c,Y_{h+1})}^{\unit}\bigl(\overline U_k\bigr)
\simeq L_{K(h+1) \vee \cdots \vee K(k)} S^0.
\]
Applying $\pi_*$ recovers the calculation of \Cref{thm:formal-chromatic}.
\end{Rem}
\section{Further examples}\label{sec:further_examples}
In this short section, we consider two further examples, from equivariant homotopy theory and modular representation theory.
\begin{Not}\label{not:cofree-borel}
Let $G$ be a finite group and $p$ a prime. We write $\Sp_{G,(p)}^{\cofree}$ for the Bousfield localization of $\Sp_{G,(p)}$ at $G_+$, or equivalently, the completion of $\Sp_{G,(p)}$ at the Thomason subset
\[
Y_G \coloneqq \supp(G_+) \subseteq \Spc(\Sp_{G,(p)}^c).
\]
As a set, the spectrum $\Spc(\Sp_{G,(p)}^c)$ and the subset $Y_G$ are computed in \cite{BalmerSanders17}.
\end{Not}

\begin{Prop}\label{prop:restriction-cofree-formal}
Let $G$ be a finite group and $p$ a prime. Restriction along the inclusion $e \le G$ defines an exact symmetric monoidal functor
\[
\res^G_e \colon \Sp_{G,(p)} \longrightarrow \Sp_{(p)}.
\]
Let $Y_G\coloneqq \supp(G_+) \subseteq \Spc(\Sp_{G,(p)}^c)$. Then $\res^G_e$ induces a morphism of Dirac ringed spaces
\[
(\varphi,\varphi^\#)\colon
\Spc(\Sp_{(p)}^c)
\longrightarrow
\Spf(\Sp_{G,(p)}^{\cofree,\dual}).
\]
Moreover, $\varphi$ is a homeomorphism of the underlying topological spaces. However, for $G=C_2$ and $p=2$, this morphism is not an isomorphism of Dirac ringed spaces.
\end{Prop}

\begin{proof}
Apply \Cref{lem:functoriality} to the restriction functor between compact objects:
\[
\res^G_e\colon\Sp_{G,(p)}^c\longrightarrow\Sp_{(p)}^c.
\]
Let $Y\coloneqq \Spc(\Sp_{(p)}^c)$. The induced continuous map is
\[
\varphi\coloneqq\Spc(\res^G_e)\colon
\Spc(\Sp_{(p)}^c)\longrightarrow\Spc(\Sp_{G,(p)}^c).
\]
By \cite[Corollary~4.13]{BalmerSanders17}, the image of $\varphi$ is $Y_G$. We claim that $\varphi$ is a homeomorphism onto this image. For $x\in\Sp_{(p)}^c$ and $Q\in\Spc(\Sp_{(p)}^c)$, the equivalence $\res_e^G\ind_e^G(x)\simeq\bigvee_{g\in G}x$ gives
\[
\ind_e^G(x)\in\varphi(Q)\iff x\in Q.
\]
Thus $\varphi$ is injective. Moreover, for every basic open $U(x)\subseteq\Spc(\Sp_{(p)}^c)$,
\[
\varphi(U(x))=Y_G\cap U(\ind_e^G(x)),
\]
which is open in the subspace topology on $Y_G$. Therefore $\varphi^{-1}$ is continuous, and $\varphi$ is a homeomorphism onto $Y_G$.
Therefore, \Cref{lem:functoriality} yields a morphism of Dirac ringed spaces
\[
(\varphi,\varphi^\#)\colon
\Spf(\Sp_{(p)}^c,Y) \longrightarrow \Spf(\Sp_{G,(p)}^c,Y_G).
\]
By \Cref{exa:whole-spectrum}, the domain is $\Spf(\Sp_{(p)}^c,Y)\cong \Spc(\Sp_{(p)}^c)$, and by \Cref{thm:completion-theorem}, the target identifies with $\Spf(\Sp_{G,(p)}^{\cofree,\dual})$, yielding the stated map. The fact that $\varphi$ is a homeomorphism on underlying spaces is proved above. 

Finally, we show this is not an isomorphism for $G=C_2$ and $p=2$ by evaluating global sections. On the nonequivariant side, the closed point of $\Spc(\Sp_{(2)}^c)$ has no proper open neighborhood, so sheafification does not change the defining presheaf value on the whole space. Hence
\[
\mathcal O_{\Spc(\Sp_{(2)}^c)}\bigl(\Spc(\Sp_{(2)}^c)\bigr)
\cong \pi_*\End_{\Sp_{(2)}}(\mathbb S_{(2)})
\cong \pi_*\mathbb S_{(2)}.
\]
In degree $0$, this is the $2$-local integers $\mathbb Z_{(2)}$.

On the cofree side (the target), the closed point of $Y_G\cong\Spc(\Sp_{(2)}^c)$ has no proper open neighborhood, so every open cover of $Y_G$ contains $Y_G$ itself. Hence sheafification does not change the defining presheaf value on the whole space, and
\[
\mathcal O_{\Spf(\Sp_{C_2,(2)}^{\cofree,\dual})}(Y_G)
\cong \pi_*\Hom_{\Sp_{C_2,(2)}}(e_{Y_G},\mathbb S_{C_2,(2)}).
\]
The Segal conjecture for $C_2$ (specifically Lin's theorem \cite[Theorem 1.1]{Lin1980Conjectures}) identifies the degree $0$ part of this ring with the $I$-adic completion of the 2-local Burnside ring $A(C_2)_{(2)}$, where $I$ is the augmentation ideal. Since the two rings are not isomorphic (for example, the latter is uncountable), the induced map on degree-$0$ global sections is not an isomorphism. Therefore, $(\varphi,\varphi^\#)$ is not an isomorphism of Dirac ringed spaces.
\end{proof}
\begin{Rem}\label{rem:modrep-example}
Our final example comes from modular representation theory. Fix a finite group $G$ and a field $k$ of characteristic $p$ dividing $|G|$. Let $\cat T\simeq\Mod_{\Sp_G}(\underline{k})$ be the stable homotopy theory associated to $k$, where $\underline{k}$ denotes the corresponding Borel-equivariant $G$-spectrum. Its homotopy category identifies with $K(\Inj kG)$. This is a symmetric monoidal, rigidly-compactly generated category, and its compact/dualizable objects identify with the bounded derived category of finitely generated modules, i.e., $\cat T^{\dual} \simeq \Der^b(kG\text{-mod})$.

By \cite{BensonCarlsonRickard97} and \cite[Proposition~8.5]{Balmer10b}, the comparison map
\[
  \rho\colon \Spc \bigl(\Der^b(kG\text{-mod})\bigr)\longrightarrow \Spec^h\bigl(H^{\bullet}(G,k)\bigr)
\]
is a homeomorphism. The homogeneous spectrum $\Spec^h(H^{\bullet}(G,k))$ has a unique closed point, corresponding to the homogeneous maximal ideal $  \mathfrak m \coloneqq H^{>0}(G,k)$. Let $Y_{\mathfrak m}\subseteq \Spc(\Der^b(kG\text{-mod}))$ denote the corresponding singleton.

The tt-completion of $\cat T$ at $Y_{\mathfrak m}$ identifies with the unbounded derived category $\Der(\Mod kG)$ equipped with the diagonal tensor product; see \cite[Example~3.29]{balmer2025tateintermediatevaluetheorem}. In this category, the dualizable objects are again $\Der^b(kG\text{-mod})$, although the compact objects need not coincide with the dualizable ones. In particular, the completed category is typically not rigidly-compactly generated. Using \Cref{prop:comparison-map} we deduce the following:
\end{Rem}

\begin{Thm}\label{thm:modrep-formal-spectrum}
With notation as above, the formal spectrum
\[
\Spf\bigl(\cat T^{\dual},Y_{\mathfrak m}\bigr)
\cong
\Spf\bigl((\widehat{\cat T}_{Y_{\mathfrak m}})^{\dual}\bigr)
\]
is a one-point Dirac ringed space, and its Dirac ring of global sections is $H^{\bullet}(G,k)^{\wedge}_{\mathfrak m}$.
\end{Thm}

\begin{proof}
By the Evens--Venkov theorem the graded ring $R_*=H^{\bullet}(G,k)$ is noetherian, so \Cref{prop:comparison-map} applies. Since $\rho$ is a homeomorphism and $\mathfrak m=H^{>0}(G,k)$ is the unique closed point of $\Spec^h(R_*)$, the subset $Y_{\mathfrak m}=\rho^{-1}(V(\mathfrak m))$ is a single point. By \Cref{prop:comparison-map} with $I=\mathfrak m$, the comparison map $\overline{\rho}$ identifies the structure sheaf of $\Spf(\cat T^{\dual},Y_{\mathfrak m})$ with that of $\Spf^h(R_*,\mathfrak m)$. The latter is the one-point space $V(\mathfrak m)=\{\mathfrak m\}$, so sheafification does not change the defining presheaf value on the whole space. Its ring of global sections is $(R_*)^{\wedge}_{\mathfrak m} = H^{\bullet}(G,k)^{\wedge}_{\mathfrak m}$. The identification $\Spf(\cat T^{\dual},Y_{\mathfrak m})\cong\Spf((\widehat{\cat T}_{Y_{\mathfrak m}})^{\dual})$ is \Cref{thm:completion-theorem}.
\end{proof}
\bibliographystyle{alpha}\bibliography{tt-geo}

@unpublished{krause2023completionstriangulatedcategories,
	title={Completions of triangulated categories}, 
	author={Henning Krause},
	year={2023},
	eprint={2309.01260},
	archivePrefix={arXiv},
	primaryClass={math.CT},
	NOTE={Available online at \url{https://arxiv.org/abs/2309.01260}}, 
}

@incollection{alonso1999duality,
	author    = {Alonso Tarr{\'i}o, Leovigildo and Jerem{\'i}as L{\'o}pez, Ana and Lipman, Joseph},
	title     = {Duality and flat base change on formal schemes},
	booktitle = {Studies in duality on noetherian formal schemes and non-noetherian ordinary schemes},
	series    = {Contemporary Mathematics},
	volume    = {244},
	pages     = {3--90},
	publisher = {American Mathematical Society},
	address   = {Providence, RI},
	year      = {1999},
	note      = {Available at \url{http://www.ams.org/books/conm/244/}}
}

@article {BarthelHeardSers2023Stratification,
	AUTHOR = {Barthel, Tobias and Heard, Drew and Sanders, Beren},
	TITLE = {Stratification in tensor triangular geometry with applications
	to spectral {M}ackey functors},
	JOURNAL = {Camb. J. Math.},
	FJOURNAL = {Cambridge Journal of Mathematics},
	VOLUME = {11},
	YEAR = {2023},
	NUMBER = {4},
	PAGES = {829--915},
	ISSN = {2168-0930,2168-0949},
	MRCLASS = {18G80 (14F08 18F99 55P42 55P91 55U35)},
	MRNUMBER = {4650265},
	DOI = {10.4310/cjm.2023.v11.n4.a2},
	URL = {https://doi.org/10.4310/cjm.2023.v11.n4.a2},
}

@article {NaumannPol2024Separable,
	AUTHOR = {Naumann, Niko and Pol, Luca},
	TITLE = {Separable commutative algebras and {G}alois theory in stable
	homotopy theories},
	JOURNAL = {Adv. Math.},
	FJOURNAL = {Advances in Mathematics},
	VOLUME = {449},
	YEAR = {2024},
	PAGES = {Paper No. 109736, 67},
	ISSN = {0001-8708},
	MRCLASS = {14A30 (13B05 14F20 18G80 20C20 55U35)},
	MRNUMBER = {4752740},
	DOI = {10.1016/j.aim.2024.109736},
	URL = {https://doi.org/10.1016/j.aim.2024.109736},
}

@incollection {BCHNPS-descent,
	AUTHOR = {Barthel, Tobias and Castellana, Nat\`alia and Heard, Drew and
	Naumann, Niko and Pol, Luca and Sanders, Beren},
	TITLE = {Descent in tensor triangular geometry},
	BOOKTITLE = {Triangulated categories in representation theory and
	beyond---the {A}bel {S}ymposium 2022},
	SERIES = {Abel Symp.},
	VOLUME = {17},
	PAGES = {1--56},
	PUBLISHER = {Springer, Cham},
	YEAR = {2024},
	MRCLASS = {18N60 (18F70 18G80 55U35)},
	MRNUMBER = {4786501},
	DOI = {10.1007/978-3-031-57789-5\_1},
	URL = {https://doi.org/10.1007/978-3-031-57789-5_1},
}

@article {BlumbergGepnerTabuada2013universal,
	AUTHOR = {Blumberg, Andrew J. and Gepner, David and Tabuada, Gon\c{c}alo},
	TITLE = {A universal characterization of higher algebraic {$K$}-theory},
	JOURNAL = {Geom. Topol.},
	FJOURNAL = {Geometry \& Topology},
	VOLUME = {17},
	YEAR = {2013},
	NUMBER = {2},
	PAGES = {733--838},
	ISSN = {1465-3060},
	MRCLASS = {19D10 (18D20 19D25 19D55 55N15 55U40)},
	MRNUMBER = {3070515},
	MRREVIEWER = {Ross Staffeldt},
	DOI = {10.2140/gt.2013.17.733},
	URL = {https://doi.org/10.2140/gt.2013.17.733},
}

@article {HironakaMatsumura1968Formal,
	AUTHOR = {Hironaka, Heisuke and Matsumura, Hideyuki},
	TITLE = {Formal functions and formal embeddings},
	JOURNAL = {J. Math. Soc. Japan},
	FJOURNAL = {Journal of the Mathematical Society of Japan},
	VOLUME = {20},
	YEAR = {1968},
	PAGES = {52--82},
	ISSN = {0025-5645},
	MRCLASS = {14.55},
	MRNUMBER = {251043},
	MRREVIEWER = {M. Miyanishi},
	DOI = {10.2969/jmsj/02010052},
	URL = {https://doi.org/10.2969/jmsj/02010052},
}

@book {GrothendieckDieudonne1971Elements,
	AUTHOR = {Grothendieck, A. and Dieudonn\'{e}, J. A.},
	TITLE = {\'{E}l\'{e}ments de g\'{e}om\'{e}trie alg\'{e}brique. {I}},
	SERIES = {Grundlehren der mathematischen Wissenschaften [Fundamental
	Principles of Mathematical Sciences]},
	VOLUME = {166},
	PUBLISHER = {Springer-Verlag, Berlin},
	YEAR = {1971},
	PAGES = {ix+466},
	ISBN = {3-540-05113-9; 0-387-05113-9},
	MRCLASS = {14-02 (14A15 14F05)},
	MRNUMBER = {3075000},
}

@book {ElmendorfKrizMellMay1997Rings,
	AUTHOR = {Elmendorf, A. D. and Kriz, I. and Mandell, M. A. and May, J.
	P.},
	TITLE = {Rings, modules, and algebras in stable homotopy theory},
	SERIES = {Mathematical Surveys and Monographs},
	VOLUME = {47},
	NOTE = {With an appendix by M. Cole},
	PUBLISHER = {American Mathematical Society, Providence, RI},
	YEAR = {1997},
	PAGES = {xii+249},
	ISBN = {0-8218-0638-6},
	MRCLASS = {55N20 (19D10 19D55 55P42 55T25)},
	MRNUMBER = {1417719},
	MRREVIEWER = {Donald M. Davis},
	DOI = {10.1090/surv/047},
	URL = {https://doi.org/10.1090/surv/047},
}

@article{HesselholtPstragowskiDiracI,
	author = {Hesselholt, Lars and Pstr{\k{a}}gowski, Piotr},
	title = {Dirac geometry {I}: Commutative algebra},
	journal = {Peking Math. J.},
	volume = {8},
	number = {3},
	year = {2025},
	pages = {405--480},
	doi = {10.1007/s42543-023-00072-6},
	url = {https://doi.org/10.1007/s42543-023-00072-6},
}

@article {bhv2,
	AUTHOR = {Barthel, Tobias and Heard, Drew and Valenzuela, Gabriel},
	TITLE = {Local duality for structured ring spectra},
	JOURNAL = {J. Pure Appl. Algebra},
	FJOURNAL = {Journal of Pure and Applied Algebra},
	VOLUME = {222},
	YEAR = {2018},
	NUMBER = {2},
	PAGES = {433--463},
	ISSN = {0022-4049},
	MRCLASS = {18G60 (55P43)},
	MRNUMBER = {3694463},
	MRREVIEWER = {Stanis\l aw Betley},
	DOI = {10.1016/j.jpaa.2017.04.012},
	URL = {https://doi.org/10.1016/j.jpaa.2017.04.012},
}

@article {bhv1,
	AUTHOR = {Barthel, Tobias and Heard, Drew and Valenzuela, Gabriel},
	TITLE = {Local duality in algebra and topology},
	JOURNAL = {Adv. Math.},
	FJOURNAL = {Advances in Mathematics},
	VOLUME = {335},
	YEAR = {2018},
	PAGES = {563--663},
	ISSN = {0001-8708},
	MRCLASS = {55U35 (14F05 55P60)},
	MRNUMBER = {3836674},
	MRREVIEWER = {Constanze Roitzheim},
	DOI = {10.1016/j.aim.2018.07.017},
	URL = {https://doi.org/10.1016/j.aim.2018.07.017},
}

@article {DwyerGreenlees02,
	AUTHOR = {Dwyer, W. G. and Greenlees, J. P. C.},
	TITLE = {Complete modules and torsion modules},
	JOURNAL = {Amer. J. Math.},
	FJOURNAL = {American Journal of Mathematics},
	VOLUME = {124},
	YEAR = {2002},
	NUMBER = {1},
	PAGES = {199--220},
	ISSN = {0002-9327},
	MRCLASS = {16E30 (16D90 18E30)},
	MRNUMBER = {1879003},
	MRREVIEWER = {Henning Krause},
	URL =
	{http://muse.jhu.edu.oca.ucsc.edu/journals/american_journal_of_mathematics/v124/124.1dwyer.pdf},
}

@article {GreenleesMay92,
	AUTHOR = {Greenlees, J. P. C. and May, J. P.},
	TITLE = {Derived functors of {$I$}-adic completion and local homology},
	JOURNAL = {J. Algebra},
	FJOURNAL = {Journal of Algebra},
	VOLUME = {149},
	YEAR = {1992},
	NUMBER = {2},
	PAGES = {438--453},
	ISSN = {0021-8693},
	MRCLASS = {13B35 (13D45 18G10 19A22)},
	MRNUMBER = {1172439},
	MRREVIEWER = {Michinori Sakaguchi},
	DOI = {10.1016/0021-8693(92)90026-I},
	URL = {https://doi-org.oca.ucsc.edu/10.1016/0021-8693(92)90026-I},
}

@incollection {Miller2002Graded,
	AUTHOR = {Miller, Ezra},
	TITLE = {Graded {Greenlees--May} duality and the {\v C}ech hull},
	BOOKTITLE = {Local cohomology and its applications},
	EDITOR = {Lyubeznik, Gennady},
	SERIES = {Lecture Notes in Pure and Appl. Math.},
	VOLUME = {226},
	PAGES = {233--253},
	PUBLISHER = {Dekker, New York},
	YEAR = {2002},
	ISBN = {0-8247-0741-9},
	URL = {https://sites.math.duke.edu/~ezra/GMCech/gmCech.pdf},
}

@article {AlonsoJeremiasLipman97,
	AUTHOR = {Alonso Tarr\'{\i}o, Leovigildo and Jerem\'{\i}as L\'{o}pez, Ana and Lipman,
	Joseph},
	TITLE = {Local homology and cohomology on schemes},
	JOURNAL = {Ann. Sci. \'{E}cole Norm. Sup. (4)},
	FJOURNAL = {Annales Scientifiques de l'\'{E}cole Normale Sup\'{e}rieure. Quatri\`eme
	S\'{e}rie},
	VOLUME = {30},
	YEAR = {1997},
	NUMBER = {1},
	PAGES = {1--39},
	ISSN = {0012-9593},
	MRCLASS = {14F99 (14B15)},
	MRNUMBER = {1422312},
	MRREVIEWER = {Gennady Lyubeznik},
	DOI = {10.1016/S0012-9593(97)89914-4},
	URL = {https://doi-org.oca.ucsc.edu/10.1016/S0012-9593(97)89914-4},
}

@article {DellAmbrogioStanley16,
	AUTHOR = {Dell'Ambrogio, Ivo and Stanley, Donald},
	TITLE = {Affine weakly regular tensor triangulated categories},
	JOURNAL = {Pacific J. Math.},
	FJOURNAL = {Pacific Journal of Mathematics},
	VOLUME = {285},
	YEAR = {2016},
	NUMBER = {1},
	PAGES = {93--109},
	ISSN = {0030-8730},
	MRCLASS = {18E30 (55P42 55U35)},
	MRNUMBER = {3554244},
	MRREVIEWER = {Robert Harold McRae},
	DOI = {10.2140/pjm.2016.285.93},
	URL = {https://doi-org.oca.ucsc.edu/10.2140/pjm.2016.285.93},
}

@article {BehrensShah2020equivariant,
	AUTHOR = {Behrens, Mark and Shah, Jay},
	TITLE = {{$C_2$}-equivariant stable homotopy from real motivic stable
	homotopy},
	JOURNAL = {Ann. K-Theory},
	FJOURNAL = {Annals of K-Theory},
	VOLUME = {5},
	YEAR = {2020},
	NUMBER = {3},
	PAGES = {411--464},
	ISSN = {2379-1683},
	MRCLASS = {14F42 (55N91 55P91 55Q91)},
	MRNUMBER = {4132743},
	MRREVIEWER = {Stephen McKean},
	DOI = {10.2140/akt.2020.5.411},
	URL = {https://doi.org/10.2140/akt.2020.5.411},
}

@article {HoveyStrickland99,
	AUTHOR = {Hovey, Mark and Strickland, Neil P.},
	TITLE = {Morava {$K$}-theories and localisation},
	JOURNAL = {Mem. Amer. Math. Soc.},
	FJOURNAL = {Memoirs of the American Mathematical Society},
	VOLUME = {139},
	YEAR = {1999},
	NUMBER = {666},
	PAGES = {viii+100},
	ISSN = {0065-9266},
	CODEN = {MAMCAU},
	MRCLASS = {55P60 (55N22 55P42 55T15)},
	MRNUMBER = {1601906},
	MRREVIEWER = {J. P. C. Greenlees},
	DOI = {10.1090/memo/0666},
	URL = {http://dx.doi.org.ep.fjernadgang.kb.dk/10.1090/memo/0666},
}

@article {BalmerSanders17,
	AUTHOR = {Balmer, Paul and Sanders, Beren},
	TITLE = {The spectrum of the equivariant stable homotopy category of a
	finite group},
	JOURNAL = {Invent. Math.},
	FJOURNAL = {Inventiones Mathematicae},
	VOLUME = {208},
	YEAR = {2017},
	NUMBER = {1},
	PAGES = {283--326},
	ISSN = {0020-9910},
	MRCLASS = {18E30 (55P42 55U35)},
	MRNUMBER = {3621837},
	MRREVIEWER = {Geoffrey M. L. Powell},
	DOI = {10.1007/s00222-016-0691-3},
	URL = {https://doi.org/10.1007/s00222-016-0691-3},
}

@incollection {Greenlees01,
	AUTHOR = {Greenlees, J. P. C.},
	TITLE = {Tate cohomology in axiomatic stable homotopy theory},
	BOOKTITLE = {Cohomological methods in homotopy theory ({B}ellaterra, 1998)},
	SERIES = {Progr. Math.},
	VOLUME = {196},
	PAGES = {149--176},
	PUBLISHER = {Birkh\"auser, Basel},
	YEAR = {2001},
	MRCLASS = {55P42 (55P60)},
	MRNUMBER = {1851253 (2004d:55008)},
	DOI = {10.1007/978-3-0348-8312-2_12},
	URL = {http://dx.doi.org.ep.fjernadgang.kb.dk/10.1007/978-3-0348-8312-2_12},
}

@article{BalmerFavi11,
	AUTHOR = {Balmer, Paul and Favi, Giordano},
	TITLE = {Generalized tensor idempotents and the telescope conjecture},
	JOURNAL = {Proc. Lond. Math. Soc. (3)},
	FJOURNAL = {Proceedings of the London Mathematical Society. Third Series},
	VOLUME = {102},
	YEAR = {2011},
	NUMBER = {6},
	PAGES = {1161--1185},
	ISSN = {0024-6115},
	MRCLASS = {18E30 (14F05 55P60)},
	MRNUMBER = {2806103 (2012d:18010)},
	URL = {http://dx.doi.org/10.1112/plms/pdq050},
}

@article {Lin1980Conjectures,
	AUTHOR = {Lin, Wen Hsiung},
	TITLE = {On conjectures of {M}ahowald, {S}egal and {S}ullivan},
	JOURNAL = {Math. Proc. Cambridge Philos. Soc.},
	FJOURNAL = {Mathematical Proceedings of the Cambridge Philosophical
	Society},
	VOLUME = {87},
	YEAR = {1980},
	NUMBER = {3},
	PAGES = {449--458},
	ISSN = {0305-0041,1469-8064},
	MRCLASS = {55Q10},
	MRNUMBER = {556925},
	MRREVIEWER = {Donald\ M.\ Davis},
	DOI = {10.1017/S0305004100056887},
	URL = {https://doi.org/10.1017/S0305004100056887},
}

@article {Stevenson2018Filtrations,
	AUTHOR = {Stevenson, Greg},
	TITLE = {Filtrations via tensor actions},
	JOURNAL = {Int. Math. Res. Not. IMRN},
	FJOURNAL = {International Mathematics Research Notices. IMRN},
	YEAR = {2018},
	NUMBER = {8},
	PAGES = {2535--2558},
	ISSN = {1073-7928,1687-0247},
	MRCLASS = {18E30},
	MRNUMBER = {3801492},
	MRREVIEWER = {Luz\ Adriana\ Mej\'{\i}a Casta\~{n}o},
	DOI = {10.1093/imrn/rnw325},
	URL = {https://doi.org/10.1093/imrn/rnw325},
}

@article {Balmer05a,
	AUTHOR = {Balmer, Paul},
	TITLE  = {The spectrum of prime ideals in tensor triangulated categories},
	JOURNAL = {J. Reine Angew. Math.},
	FJOURNAL = {Journal f\"ur die Reine und Angewandte Mathematik},
	VOLUME = {588},
	YEAR = {2005},
	PAGES = {149--168},
}

@article {Balmer02,
	AUTHOR = {Balmer, P.},
	TITLE = {Presheaves of triangulated categories and reconstruction of
	schemes},
	JOURNAL = {Math. Ann.},
	FJOURNAL = {Mathematische Annalen},
	VOLUME = {324},
	YEAR = {2002},
	NUMBER = {3},
	PAGES = {557--580},
	ISSN = {0025-5831},
}

@article{Balmer10b,
	AUTHOR = {Balmer, Paul},
	TITLE  = {Spectra, spectra, spectra -- Tensor triangular spectra versus {Z}ariski spectra of endomorphism rings},
	JOURNAL = {Algebr. Geom. Topol.},
	FJOURNAL = {Algebraic \& Geometric Topology},
	VOLUME = {10},
	YEAR = {2010},
	NUMBER = {3},
	PAGES = {1521--1563},
	MRCLASS = {18E30 (14F05, 19K35, 20C20, 55P42, 55U35)},
}

@misc {Aoki2025Sheaves,
    AUTHOR = {Aoki, Ko},
     TITLE = {The sheaves--spectrum adjunction},
      YEAR = {2025},
    EPRINT = {2302.04069},
ARCHIVEPREFIX = {arXiv},
PRIMARYCLASS = {math.CT},
       DOI = {10.48550/arXiv.2302.04069},
       URL = {https://arxiv.org/abs/2302.04069},
}

@incollection {ThomasonTrobaugh90,
	AUTHOR = {Thomason, R. W. and Trobaugh, T.},
	TITLE = {Higher algebraic {$K$}-theory of schemes and of derived
	categories},
	BOOKTITLE = {The Grothendieck Festschrift, Vol.\ III},
	SERIES = {Progr. Math.},
	VOLUME = {88},
	PAGES = {247--435},
	PUBLISHER = {Birkh\"auser},
	ADDRESS = {Boston, MA},
	YEAR = {1990},
}

@article {Thomason97,
	AUTHOR = {Thomason, R. W.},
	TITLE = {The classification of triangulated subcategories},
	JOURNAL = {Compositio Math.},
	FJOURNAL = {Compositio Mathematica},
	VOLUME = {105},
	YEAR = {1997},
	NUMBER = {1},
	PAGES = {1--27},
}

@article {BensonCarlsonRickard97,
	AUTHOR = {Benson, Dave and Carlson, Jon F. and Rickard, Jeremy},
	TITLE = {Thick subcategories of the stable module category},
	JOURNAL = {Fund. Math.},
	FJOURNAL = {Fundamenta Mathematicae},
	VOLUME = {153},
	YEAR = {1997},
	NUMBER = {1},
	PAGES = {59--80},
	ISSN = {0016-2736},
}

@incollection {Hopkins87,
	AUTHOR = {Hopkins, Michael J.},
	TITLE = {Global methods in homotopy theory},
	BOOKTITLE = {Homotopy theory (Durham, 1985)},
	SERIES = {LMS Lect. Note},
	VOLUME = {117},
	PAGES = {73--96},
	PUBLISHER = {Cambridge Univ. Press},
	YEAR = {1987},
}

@article {Neeman92a,
	AUTHOR = {Neeman, Amnon},
	TITLE = {The chromatic tower for {$D(R)$}},
	JOURNAL = {Topology},
	FJOURNAL = {Topology. An International Journal of Mathematics},
	VOLUME = {31},
	YEAR = {1992},
	NUMBER = {3},
	PAGES = {519--532},
}

@book {HTTLurie,
	AUTHOR = {Lurie, Jacob},
	TITLE = {Higher topos theory},
	SERIES = {Annals of Mathematics Studies},
	VOLUME = {170},
	PUBLISHER = {Princeton University Press, Princeton, NJ},
	YEAR = {2009},
	PAGES = {xviii+925},
	ISBN = {978-0-691-14049-0; 0-691-14049-9},
	MRCLASS = {18-02 (18B25 18E35 18G30 18G55 55U40)},
	MRNUMBER = {2522659},
	MRREVIEWER = {Mark Hovey},
	DOI = {10.1515/9781400830558},
	URL = {https://doi.org/10.1515/9781400830558},
}

@unpublished{HALurie,
	Author = {Lurie, Jacob},
	Title = {Higher Algebra},
	NOTE = {1553~pages, available from the author's website},
	Year = {2017},
}

@article {HoveyPalmieriStrickland97,
	AUTHOR = {Hovey, Mark and Palmieri, John H. and Strickland, Neil P.},
	TITLE = {Axiomatic stable homotopy theory},
	JOURNAL = {Mem. Amer. Math. Soc.},
	FJOURNAL = {Memoirs of the American Mathematical Society},
	VOLUME = {128},
	YEAR = {1997},
	NUMBER = {610},
}

@article {BuanKrauseSolberg07,
	AUTHOR = {Buan, Aslak Bakke and Krause, Henning and Solberg, {\O}yvind},
	TITLE = {Support varieties: an ideal approach},
	JOURNAL = {Homology, Homotopy Appl.},
	FJOURNAL = {Homology, Homotopy and Applications},
	VOLUME = {9},
	YEAR = {2007},
	NUMBER = {1},
	PAGES = {45--74},
	ISSN = {1532-0073},
	MRCLASS = {18E30},
	MRNUMBER = {MR2280286},
}

@article {Neeman92b,
	AUTHOR = {Neeman, Amnon},
	TITLE = {The connection between the {$K$}-theory localization theorem
	of {T}homason, {T}robaugh and {Y}ao and the smashing
	subcategories of {B}ousfield and {R}avenel},
	JOURNAL = {Ann. Sci. \'Ecole Norm. Sup. (4)},
	FJOURNAL = {Annales Scientifiques de l'\'Ecole Normale Sup\'erieure.
	Quatri\`eme S\'erie},
	VOLUME = {25},
	YEAR = {1992},
	NUMBER = {5},
	PAGES = {547--566},
	ISSN = {0012-9593},
	CODEN = {ASENAH},
	MRCLASS = {18E30 (19D10 19E08)},
	MRNUMBER = {MR1191736 (93k:18015)},
	MRREVIEWER = {Steven E. Landsburg},
}

@article {BenMosheCarmeliSchlankYanovski2025Descent,
	AUTHOR = {Ben-Moshe, Shay and Carmeli, Shachar and Schlank, Tomer and
	Yanovski, Lior},
	TITLE = {Descent and cyclotomic redshift for chromatically localized
	algebraic {$K$}-theory},
	JOURNAL = {J. Amer. Math. Soc.},
	FJOURNAL = {Journal of the American Mathematical Society},
	VOLUME = {38},
	YEAR = {2025},
	NUMBER = {2},
	PAGES = {521--583},
	ISSN = {0894-0347,1088-6834},
	MRCLASS = {19D10 (55P42)},
	MRNUMBER = {4868950},
	DOI = {10.1090/jams/1052},
	URL = {https://doi.org/10.1090/jams/1052},
}

@article{BarthelHeardNaumann20pp,
	AUTHOR = {Barthel, Tobias and Heard, Drew and Naumann, Niko},
	TITLE = {On conjectures of {H}ovey-{S}trickland and {C}hai},
	JOURNAL = {Selecta Math. (N.S.)},
	FJOURNAL = {Selecta Mathematica. New Series},
	VOLUME = {28},
	YEAR = {2022},
	NUMBER = {3},
	ISSN = {1022-1824},
	MRCLASS = {14L05 (11S31 55N22 55P42)},
	MRNUMBER = {4405748},
	DOI = {10.1007/s00029-022-00766-2},
	URL = {https://doi.org/10.1007/s00029-022-00766-2},
}

@article {Heard2023local,
	AUTHOR = {Heard, Drew},
	TITLE = {The {${\rm Sp}_{k,n}$}-local stable homotopy category},
	JOURNAL = {Algebr. Geom. Topol.},
	FJOURNAL = {Algebraic \& Geometric Topology},
	VOLUME = {23},
	YEAR = {2023},
	NUMBER = {8},
	PAGES = {3655--3706},
	ISSN = {1472-2747},
	MRCLASS = {55P42 (55P60 55T15)},
	MRNUMBER = {4665280},
	DOI = {10.2140/agt.2023.23.3655},
	URL = {https://doi.org/10.2140/agt.2023.23.3655},
}

@incollection {GreenleesMay1995Completions,
	AUTHOR = {Greenlees, J. P. C. and May, J. P.},
	TITLE = {Completions in algebra and topology},
	BOOKTITLE = {Handbook of algebraic topology},
	PAGES = {255--276},
	PUBLISHER = {North-Holland, Amsterdam},
	YEAR = {1995},
	MRCLASS = {55P60},
	MRNUMBER = {1361892},
	MRREVIEWER = {Jean-Claude Thomas},
	DOI = {10.1016/B978-044481779-2/50008-0},
	URL = {https://doi.org/10.1016/B978-044481779-2/50008-0},
}

@article {abhs,
	AUTHOR = {Arone, Gregory and Barthel, Tobias and Heard, Drew and
	Sanders, Beren},
	TITLE = {The spectrum of excisive functors},
	JOURNAL = {Invent. Math.},
	FJOURNAL = {Inventiones Mathematicae},
	VOLUME = {241},
	YEAR = {2025},
	NUMBER = {2},
	PAGES = {363--464},
	ISSN = {0020-9910},
	MRCLASS = {Prelim},
	MRNUMBER = {4929647},
	DOI = {10.1007/s00222-025-01338-9},
	URL = {https://doi.org/10.1007/s00222-025-01338-9},
}

@article{balmer2025tateintermediatevaluetheorem,
	title = {The {T}ate {I}ntermediate {V}alue {T}heorem},
	journal = {Advances in Mathematics},
	volume = {483},
	pages = {110675},
	year = {2025},
	issn = {0001-8708},
	doi = {https://doi.org/10.1016/j.aim.2025.110675},
	url = {https://www.sciencedirect.com/science/article/pii/S0001870825005730},
	author = {Paul Balmer and Beren Sanders},
}

@book {MayPonto2012More,
	AUTHOR = {May, J. P. and Ponto, K.},
	TITLE = {More concise algebraic topology},
	SERIES = {Chicago Lectures in Mathematics},
	NOTE = {Localization, completion, and model categories},
	PUBLISHER = {University of Chicago Press, Chicago, IL},
	YEAR = {2012},
	PAGES = {xxviii+514},
	ISBN = {978-0-226-51178-8; 0-226-51178-2},
	MRCLASS = {55-02 (16T05 18G55 55P60)},
	MRNUMBER = {2884233},
	MRREVIEWER = {Ismar Voli\'{c}},
}

@unpublished{balmer2024perfectcomplexescompletion,
	title={Perfect complexes and completion}, 
	author={Paul Balmer and Beren Sanders},
	year={2024},
	eprint={2411.14761},
	archivePrefix={arXiv},
	primaryClass={math.AC},
	url={https://arxiv.org/abs/2411.14761},
	NOTE = {Preprint, 20~pages, available online at \href{https://arxiv.org/pdf/2411.14761}{arxiv:2411.14761}},
}

@article {MR5009505,
	AUTHOR = {Calm\`es, Baptiste and Harpaz, Yonatan and Land, Markus and
	Nardin, Denis and Steimle, Wolfgang and Dotto, Emanuele and
	Hebestreit, Fabian and Moi, Kristian and Nikolaus, Thomas},
	TITLE = {Hermitian {K}-theory for stable {$\infty$}-categories {II}:
	{C}obordism categories and additivity},
	JOURNAL = {Acta Math.},
	FJOURNAL = {Acta Mathematica},
	VOLUME = {235},
	YEAR = {2026},
	NUMBER = {2},
	PAGES = {149--400},
	ISSN = {0001-5962,1871-2509},
	MRCLASS = {19G38 (18N99 55U35)},
	MRNUMBER = {5009505},
	DOI = {10.4310/acta.2025.n235.n2.a1},
	URL = {https://doi.org/10.4310/acta.2025.n235.n2.a1},
}

@article {ErnePicadoPultr2022Adjoint,
    AUTHOR = {Ern\'{e}, Marcel and Picado, Jorge and Pultr, Ale\v{s}},
     TITLE = {Adjoint maps between implicative semilattices and continuity
              of localic maps},
   JOURNAL = {Algebra Universalis},
  FJOURNAL = {Algebra Universalis},
    VOLUME = {83},
      YEAR = {2022},
    NUMBER = {2},
     PAGES = {Paper No. 13, 23},
      ISSN = {0002-5240,1420-8911},
   MRCLASS = {06D20 (06D22 18F70 54C05)},
  MRNUMBER = {4400114},
MRREVIEWER = {Tomasz\ Kubiak},
       DOI = {10.1007/s00012-022-00767-4},
       URL = {https://doi.org/10.1007/s00012-022-00767-4},
}

@misc{aoki2025higherzariskigeometry,
	title={Higher {Z}ariski Geometry}, 
	author={Ko Aoki and Tobias Barthel and Anish Chedalavada and Tomer Schlank and Greg Stevenson},
	year={2025},
	eprint={2508.11621},
	archivePrefix={arXiv},
	primaryClass={math.AG},
	url={https://arxiv.org/abs/2508.11621}, 
}
\end{document}